\documentclass[leqno,12pt]{amsart}

\usepackage{amssymb}
\usepackage{amsmath}
\usepackage[usenames, dvipsnames]{color}
\usepackage{enumitem}
\usepackage{amsthm}
\usepackage{mathrsfs}

\numberwithin{equation}{section}

\usepackage{marginnote}
\usepackage{hyperref}
\usepackage{esint}
\usepackage{verbatim}

\usepackage{hyperref}

\usepackage{geometry} 

\usepackage{etoolbox}

\theoremstyle{plain}
\newtheorem{theorem}{Theorem}[section]
\newtheorem{lemma}[theorem]{Lemma}

\newtheorem{definition}[theorem]{Definition}

\newtheorem{example}[theorem]{Example}
\newtheorem{assumption}[theorem]{Assumption}
\newtheorem{remark}[theorem]{Remark}

\newcommand{\R}{\mathbb{R}}

\def\XXint#1#2#3{{\setbox0=\hbox{$#1{#2#3}{\int}$ }
\vcenter{\hbox{$#2#3$ }}\kern-.6\wd0}}

\title[ABP Estimates for Singular-Degenerate Equations]{On ABP Estimates for a Class of Fully Nonlinear Singular-Degenerate Elliptic Equations and Applications}

\author[J. Fang]{Junyuan Fang}
\address[J. Fang]{Department of Mathematics, Louisiana State University, 125B Prescott Hall,
Baton Rouge, LA 70803-4918}
\email{jfang8@lsu.edu}

\author[J. Peltier]{Jack Peltier}
\address[J. Peltier]{Division of Applied Mathematics, Box F 182 George Street, 
Brown University, Providence, RI 02912}
\email{jack\_peltier@brown.edu}

\author[T. Phan]{Tuoc Phan}
\address[T. Phan]{Department of Mathematics, University of Tennessee, 227 Ayres Hall,
1403 Circle Drive, Knoxville, TN 37996-1320}
\email{tphan2@utk.edu}

\subjclass[2020]{35B50, 35B45, 35J60, 35J96}
\keywords{ABP estimate, singular-degenerate coefficients, fully nonlinear equations, anisotropic minimal surface equations, Monge-Amp\`ere type equations}

\begin{document}
\begin{abstract} We establish an Alexandrov-Bakelman-Pucci type estimate for a class of fully nonlinear elliptic equations with singular or degenerate coefficients and nonlinear dependence on the gradient. The estimate bounds sub-solutions in terms of their boundary values and weighted integrals of the data. It holds for classical sub-solutions and for viscosity sub-solutions with locally Lipschitz gradients. The assumptions allow general rates of growth in the gradient, with a smallness condition on the data when the associated gradient weight has finite integral. The proof uses the classical weighted contact-set argument. As applications, we obtain sup-norm estimates for viscosity sub-solutions to a class of generalized Monge-Amp\`ere equations, Hessian equations, Hessian quotients under an additional condition on the contact Hessian, and anisotropic mean-curvature-type equations.
\end{abstract}
\maketitle
\section{Introduction and main results} Let  $\Omega \subset\mathbb{R}^n$ be an open bounded domain with boundary $\partial \Omega$. We study Alexandrov-Bakelman-Pucci (ABP) estimates for fully nonlinear equations of the form
\begin{equation}\label{main eqn}
    F(x, u(x), D u(x), D^2u(x)) + G(x, u(x), Du(x))= f(x)\quad \text{in}\quad \Omega,
\end{equation}
where $u: \Omega \rightarrow \mathbb{R}$ is the unknown function, $Du$ and $D^2u$ denote its gradient and Hessian, respectively, and $f : \Omega \rightarrow \mathbb{R}$ is a given continuous function. The operators $F$ and $G$ are defined as
\begin{align*}\label{operator F}
& \begin{array}{cccl}
F: & \Omega\times \R \times \mathbb{R}^n  \times S(n) &\longrightarrow &\mathbb{R}\\[4pt]
&(x, z, \xi, X)\quad &\longmapsto &F(x, z, \xi, X), \quad \text{and}
\end{array} \\
& \begin{array}{cccl}
G: & \Omega\times \R \times \mathbb{R}^n   &\longrightarrow &\mathbb{R}\\[4pt]
&(x, z, \xi)\quad &\longmapsto &G(x, z, \xi),
\end{array} 
\end{align*}
in which $S(n)$ denotes the set of $n \times n$ symmetric matrices.
A key feature of the equation \eqref{main eqn} we are considering is that it could be singular or degenerate in the sense that the leading operator $F$ loses uniform coercivity and its coefficients become unbounded. For example, in the scalar model $-a(x)\Delta u$, with $a(x)>0$, degeneracy and singularity correspond to $a(x)$ approaching zero and becoming unbounded, respectively.

\smallskip
Our structural assumptions on \eqref{main eqn}, \textup{Assumption \ref{F-assumption}} below, express the required determinant lower bound through a locally integrable spatial weight $\mu$ and a gradient profile $g_1$, while a second profile $g_2$ controls the allowable growth of the lower-order term $G$. Under these assumptions, we obtain an ABP estimate for classical subsolutions and for viscosity subsolutions with locally Lipschitz gradients. The bound depends on weighted norms of the data and on the integral of $(g_1g_2)^{p_0}$ in gradient space.

\smallskip
The ABP estimates give $L^{\infty}$ bound for sub-solutions in terms of their boundary values and an integral of the source term. Its classical form for linear elliptic equations in non-divergence form was developed in \cite{Alek, Ln_drift, Bakelman, Pucci}. Moreover, it is an essential tool in the Krylov--Safonov theory and in the regularity theory of fully nonlinear elliptic equations; see \cite{CC-1, CIL, GT, Krylov-book, KS, KS1}. Extensions involving unbounded drift and weaker integrability assumptions have been studied in \cite{Cho, DK, Krylov-0, Nazarov}. Gradient dependence and degeneracy lead to further forms of the estimate; see  \cite{ARV,ABP_divform,Imbert,HP}, for example. 

\smallskip
The classical way to prove the ABP estimate for solutions to linear and quasilinear uniformly elliptic equations in non-divergence form can be described as two steps. The first part is by a geometric intrinsic inclusion observation for general continuous functions. Roughly speaking, for a function $u$ defined in a bounded domain $\Omega$, one has 
\[
B_r(0)\subseteq Du\bigl(\Gamma(u)\bigr) \quad \quad \text{for all } \quad
0<r<\frac{\sup\nolimits_{\Omega}u^+-\sup\nolimits_{\partial\Omega}u^+}{\textup{diam}(\Omega)},
\]
where $Du\bigl(\Gamma(u)\bigr)$ denotes the image of the gradient map of $u$ on the upper contact set $\Gamma(u)$. See Subsection~\ref{contact section} below or \cite[Section~9.1]{GT} for precise details. Next, the second part is to apply the area formula to the gradient map of $u$. In this step, the equation of the solution is involved to control the Jacobian determinant $|\det D^2u|$. Connecting these two will give an estimate for $\sup_\Omega u^+$ in terms of its boundary values and an integral of the source term.


\smallskip
In this paper, we retain the geometric inclusion described above and adapt the second step to fully nonlinear equations with spatial weights and nonlinear lower-order terms. We formulate a weighted contact-set argument through a determinant lower bound for a fully nonlinear operator and combine it with spatially weighted data and nonlinear lower-order terms. The structural inequality for $F$ controls the determinant of the negative Hessian on the upper contact set, while the weight associated with $G$ controls its growth in the gradient. The structural conditions are required only when the gradient is larger than a prescribed threshold. Applying the area formula to the gradient map then gives an integral inequality, which we invert to obtain the required bound; the excluded small gradients contribute an explicit term to the estimate. This formulation allows general rates of gradient growth, provided the stated structural assumptions hold. When the total integral of the gradient weight is finite, the data must satisfy the smallness condition in Theorems \ref{ABP thm} and \ref{G-ABP-thm}. When that integral is infinite, this condition is automatic. Thus the admissible size of the data depends on both the determinant bound for $F$ and the growth of $G$.

\smallskip
The viscosity result is proved under the additional assumption $u\in C^{1,1}_{\mathrm{loc}}(\Omega)$. This regularity is assumed, rather than obtained from the equation. It permits the use of the area formula for the gradient and of upper test functions at almost every point of the contact set. The estimates in \cite{Imbert, ABP_divform, HP} allow nonsmooth viscosity solutions under their respective structural assumptions. Our result concerns a different combination of coefficient and gradient dependence, with the additional solution regularity just described. Moreover, the continuity assumed for $F$ in the viscosity formulation includes zero gradients; the gradient threshold in the structural condition does not remove this continuity requirement. 

\smallskip
The applications illustrate how the choice of gradient weights reflects the structure of the equation. For Hessian equations, the Newton--Maclaurin inequality supplies the determinant bound needed in the contact-set argument. For Hessian quotients with a nonconstant denominator, we impose an additional bound on the ratio of the largest and smallest eigenvalues of $-D^2u$ at contact points where this matrix is positive definite. This condition is used to obtain a determinant lower bound for the quotient. For mean-curvature-type equations, we use a symmetric anisotropic coefficient matrix and compute the relevant integrals in gradient space to obtain explicit bounds under a smallness condition. We also treat a normalized quadratic term in the gradient. Finally, for generalized Monge-Amp\`ere equations, the leading coefficient may depend on the point, the solution, and its gradient, while the lower-order term has growth of order at most the dimension. In this case, the associated integral in gradient space is infinite, and the resulting bound does not require smallness of the data. These applications include recoveries of familiar estimates and bounds with additional coefficient and gradient dependence. The precise assumptions and estimates are given in Section \ref{mean-curve}.

\smallskip
To state the main results, let us introduce some definitions and assumptions. For each $p \in [1, \infty)$, the weighted Lebesgue space $L^p(\Omega, \mu)$ is the space consisting of all measurable functions $f: \Omega \rightarrow \mathbb{R}$ such that
\[
\|f\|_{L^p(\Omega, \mu)} = \left( \int_{\Omega} |f(x)|^p \mu(x) dx\right)^{1/p} < \infty.
\]
Similarly, $W^{k,p}(\Omega)$ denotes the standard Sobolev space for $k = 1, 2$. In addition, for each $r>0$ and $x_0 \in \mathbb{R}^n$, $B_r(x_0)$ denotes the ball of radius $r$ centered at $x_0$. We set $B_0(x_0)=\emptyset$. We now state the assumptions on the functions $F$ and $G$.

\begin{assumption} \label{F-assumption} Let $F:  \Omega \times \mathbb{R} \times \mathbb{R}^n \times S(n)\to \mathbb R$ and $G:  \Omega \times \mathbb{R} \times \mathbb{R}^n\to\mathbb R$ be measurable functions.  We assume that there exist a locally integrable function $\mu: \mathbb{R}^n \rightarrow (0, \infty)$, two bounded and measurable functions $g_k: \mathbb{R}^n \rightarrow (0, \infty)$  with $k =1,2$, and two numbers $p_0 \in [1, \infty)$ and $\gamma_0 \geq 0$ such that the following two assumptions hold.
\begin{enumerate}[label=\normalfont(\roman*)]
\item \label{F-assumption-2} For every $\phi \in C^{2}_{\text{loc}}(\Omega)$ and for $x \in \Omega$ so that $D^2\phi(x)$ is non-positive definite  and $|D \phi(x)| \geq \gamma_0$, we have
 \begin{align*}
 & g_1(D \phi(x)) (\textsf{det}(-D^2\phi(x)))^{\frac{1}{p_0}}   \leq  \mu(x)^{\frac{1}{p_0}} F(x, \phi(x), D\phi(x), D^2 \phi(x)),
 \end{align*}
 where $ \textsf{det}(\cdot)$ denotes the determinant operator acting on $n\times n$ matrices. 
\item \label{F-assumption-3} There exists a non-negative function $b \in L^{p_0}(\Omega, \mu)$ such that
  \[
  g_2(\xi) |G(x, z, \xi)| \leq b(x),  \quad \forall \ (x, z, \xi)\in \Omega \times \R \times (\mathbb{R}^n \setminus B_{\gamma_0}(0)).
  \]
\end{enumerate}
\end{assumption}
\noindent
Condition \ref{F-assumption-2} could equivalently be stated for every $(x,z,\xi,X)$ with $X\leq0$ and $|\xi|\geq\gamma_0$. Indeed, any such value, gradient, and Hessian can be realized at $x$ by a quadratic polynomial. We use this equivalent formulation when evaluating the condition at almost every contact point of a viscosity sub-solution.

\smallskip
In Assumption \ref{F-assumption}-\ref{F-assumption-2}, the function $g_1(\xi)$ measures a lower bound for the determinant coercivity of $F(x,z,\xi,X)$ when $X\leq0$. Similarly, from the Assumption \ref{F-assumption}-\ref{F-assumption-3}, $\frac{1}{g_2(\xi)}$ is the growth in $\xi \in \mathbb{R}^n \setminus B_{\gamma_0}(0)$ of the operator $G(x, z, \xi)$. For example, let us consider the following class of Hamilton--Jacobi equations
\begin{equation} \label{model-exm}
-a_{ij}(x, u(x)) D_{ij}u(x)+ G(x, u(x), Du(x)) = f(x) \quad \text{in} \quad \Omega
\end{equation}
in which $(a_{ij})_{i,j=1}^n : \Omega \times \mathbb{R} \rightarrow \R^{n\times n}$ is a non-negative symmetric matrix of continuous functions, $D_{ij} u(x)$ denotes the second partial derivative with respect to $x_i$ and $x_j$ variables of the unknown solution $u$, and the Einstein summation convention is used. We assume that there is a locally integrable function $\mu: \mathbb R^n \rightarrow (0, \infty)$ such that
\[
 1 \leq \mu(x) \textsf{det}(a_{ij}(x, z)), \quad \forall \ x \in \Omega,\, \forall \ z \in \mathbb{R}.
\]
Then, it follows that  Assumption \ref{F-assumption}-\ref{F-assumption-2} holds with $p_0 =n, g_1(\xi) =1$ due to the inequality of arithmetic and geometric means. On the other hand, if there is $\alpha \geq 0$ so that the function $G :\Omega \times \mathbb{R} \times \mathbb{R}^n \rightarrow \mathbb{R}$ satisfies
\[
|G(x, z, \xi)| \leq b(x) (|\xi|^{\alpha} +1), \quad \forall (x, z, \xi) \in \Omega \times \mathbb{R} \times (\mathbb{R}^n \setminus B_{\gamma_0}).
\]
Then,  Assumption \ref{F-assumption}-\ref{F-assumption-3} holds with $g_2(\xi) = (|\xi|^\alpha +1)^{-1}$ as long as $b \in L^{n}(\Omega, \mu)$. We emphasize that there is no restriction on $\alpha$. For $\alpha>1$, the corresponding gradient weight is integrable to the power $n$, so the ABP estimate also requires the smallness condition in the theorems below. 

\smallskip
To state our main results, let $g: \mathbb{R}^n \rightarrow (0, \infty)$ be defined by $g(\xi) = g_1(\xi) g_2(\xi)$, where $g_1, g_2$ are defined in \textup{Assumption \ref{F-assumption}}. Then, it follows that the map $r \mapsto  \|g\|_{L^{p_0}(B_r(0))}$ is continuous and strictly increasing on $(0, \infty)$ as $|\{ \xi \in \R^n: g_1(\xi) g_2(\xi) =0\}| =0$. Hence, its inverse exists and is denoted by $H:\left[0,\|g\|_{L^{p_0}(\mathbb R^n)}\right)\to[0,\infty)$. We set $H(0)=0$. In particular, it satisfies
\begin{equation} \label{H-def}
 H(\|g\|_{L^{p_0}(B_r(0))}) =r, \quad \forall\ r \in (0, \infty).
\end{equation}
We emphasize that here, it is possible that
\[
\|g\|_{L^{p_0}(\mathbb{R}^n)} = \infty.
\]
Our first result is on the maximum principle for the classical solutions to  \eqref{main eqn}.
\begin{theorem}\label{ABP thm} Suppose that $\Omega \subset \mathbb{R}^n$ is a non-empty bounded open set in $\R^n$ with boundary $\partial \Omega$, and  \textup{Assumption \ref{F-assumption}} holds. Let $f: \Omega \rightarrow \R$ be a measurable function satisfying $f^+\in L^{p_0}(\Omega, \mu)$ with $p_0 \in [1, \infty)$ the number defined in \textup{Assumption \ref{F-assumption}}, and  let 
\begin{equation}  \label{gamma-def}
\begin{split}
\gamma  & =   2^{\frac{p_0-1}{p_0}} \Big(|B_{\gamma_0}(0)|\cdot\|g_1 g_2\|_{L^\infty(\mathbb{R}^n)}^{p_0} \\
& \qquad \qquad +\|g_2\|^{p_0}_{L^\infty(\mathbb{R}^n)} \|f^+\|^{p_0}_{L^{p_0}(\Omega, \mu)}  + \|b\|^{p_0}_{L^{p_0}(\Omega, \mu)} \Big)^{\frac{1}{p_0}},
\end{split}
\end{equation}
where $\gamma_0\in [0, \infty)$ and the functions $g_1,\, g_2,\, b$, and $\mu$ are defined in \textup{Assumption \ref{F-assumption}}. If $\gamma < \|g_1g_2\|_{L^{p_0}(\mathbb{R}^n)}$, then
\begin{equation}\label{abp}
 \sup_{\Omega} u^+ \le \sup_{\partial\Omega} u^+ + \textup{diam}(\Omega) \max\big\{ \gamma_0, H(\gamma) \big\},
\end{equation}
for every $u \in C^2(\Omega) \cap C(\overline{\Omega})$ satisfying 
\[
F(x,u,Du,D^2u)+G(x,u,Du)\leq f(x)
\] 
pointwise in $\Omega$, where $H$ is defined in \eqref{H-def}.
\end{theorem}

\smallskip
In this paper, we also prove that Theorem \ref{ABP thm} holds for viscosity sub-solutions with locally Lipschitz gradients. To this end, let us recall the definition of viscosity solutions. We refer to \cite{CC-1, CIL, Kat, Krylov-book, Silvestre} for more details on viscosity solutions.
\begin{definition}\label{vis soln def} Suppose that $F \in C(\Omega\times \R \times \mathbb{R}^n  \times S(n))$, $G \in C(\Omega\times \R \times \mathbb{R}^n)$ and $f \in C(\Omega)$. A function $u \in C(\Omega)$ is said to be a viscosity sub-solution of \eqref{main eqn} in $\Omega$ if for all $\varphi\in C^2(\Omega)$, and $x\in \Omega$ at which $u-\varphi$ has a local maximum, one has
\[
F(x, u(x), D \varphi(x), D^2\varphi(x)) + G(x, u(x), D\varphi(x))\leq f(x).
\] 
Similarly, a function $u \in C(\Omega)$ is called a viscosity super-solution of \eqref{main eqn} in $\Omega$  if for all $\varphi\in C^2(\Omega)$, and $x\in \Omega$ at which $u-\varphi$ has a local minimum, one has 
\[
F(x, u(x), D \varphi(x), D^2\varphi(x)) + G(x, u(x), D\varphi(x))\geq f(x).
\] 
Finally, $u \in C(\Omega)$ is called a viscosity solution of \eqref{main eqn} in $\Omega$ if it is both a viscosity sub-solution and a viscosity super-solution. 
\end{definition}
The following result on the ABP-estimates for viscosity solutions is another main result of the paper.
\begin{theorem}\label{G-ABP-thm} Suppose that $\Omega \subset \mathbb{R}^n$ is a non-empty open bounded set in $\mathbb{R}^n$ with boundary $\partial \Omega$, and suppose that $F \in C(\Omega\times \R \times \mathbb{R}^n  \times S(n))$, $G \in C(\Omega\times \R \times \mathbb{R}^n)$ and \textup{Assumption \ref{F-assumption}} holds. Let $f \in C(\Omega)$ such that $f^+ \in L^{p_0}(\Omega, \mu),$ and let $\gamma\ge0$ be defined in \eqref{gamma-def}. If $\gamma < \|g_1g_2\|_{L^{p_0}(\mathbb{R}^n)}$, then 
\begin{align}\label{g-abp}
 \sup_{\Omega} u^+ \le
\sup_{\partial\Omega} u^+  +  \textup{diam}(\Omega)\max \big\{ \gamma_0, H(\gamma) \big\},
\end{align}
for every viscosity sub-solution $u \in C(\overline{\Omega})\cap C^{1,1}_{\mathrm{loc}}(\Omega)$ to equation \eqref{main eqn}, where $H$ is defined in \eqref{H-def}, and the functions $\mu, g_1, g_2$ and the numbers $\gamma_0,  p_0$ are defined in \textup{Assumption \ref{F-assumption}}.
\end{theorem}

\noindent
For comparison with \cite[Theorem 10.5]{GT}, consider $p_0=n$, $\gamma_0=0$, and $F(x,z,\xi,X)=-\operatorname{Tr}(A(x,z,\xi)X)$ with $A$ positive definite. The trace-determinant inequality shows that the first structural condition is equivalent to
\[
g_1(\xi)\leq n\mu(x)^{\frac{1}{n}}\det(A(x,z,\xi))^{\frac{1}{n}}.
\]
Taking the gradient profile $g_1g_2$ and the spatial datum $\mu^{\frac{1}{n}}(\|g_2\|_\infty f^++b)$ in the classical weighted argument gives the corresponding quasilinear bound. Thus the inverse-integral form of the estimate agrees with that formulation. Here the determinant condition is imposed directly on $F$, without requiring linear dependence on the Hessian. We also note that the condition $\gamma < \|g_1g_2\|_{L^{p_0}(\mathbb{R}^n)}$ in Theorem \ref{ABP thm} and Theorem \ref{G-ABP-thm} is trivial  when
\[ \|g_1 g_2\|_{L^{p_0}(\mathbb{R}^{n})} = \infty.
\]

\smallskip
The paper is organized as follows. In Section 2, we recall the properties of upper contact sets and prove Theorem \ref{ABP thm}. In Section 3, we discuss concave envelopes and prove Theorem \ref{G-ABP-thm}. Section \ref{mean-curve} contains the applications to Hessian equations, Hessian quotients, mean-curvature-type equations, and generalized Monge-Amp\`ere equations. The mean curvature estimates are related to the estimates for divergence form operators in \cite{ABP_divform}.

\section{Classical Solutions and Proof of Theorem \ref{ABP thm}}
\subsection{Contact Sets and Their Properties}\label{contact section} In this subsection, we recall some definitions and properties of contact sets of continuous functions that are needed in the proof of Theorem \ref{ABP thm}. To begin with, let $\Omega \subset \mathbb{R}^n$ be an open non-empty set. For a given $u \in C(\Omega)$, we denote the upper contact set of $u$ in $\Omega$ by
\[
\Gamma(u)=\left\{x\in\Omega : 
\begin{array}{l}
\text{there exists } \xi=\xi(x)\in\mathbb{R}^n \text{ such that} \\
u(y)\le u(x)+\langle \xi,y-x\rangle \quad \text{for all } y\in\Omega
\end{array}
\right\}.
\]
The following lemma characterizing the upper contact set is well-known.
\begin{lemma}\label{UCS D2u} If $u\in C^2(\Omega)$, then $D^2 u(x_0)$ is non-positive definite for every $x_0\in\Gamma(u)$.
\end{lemma}

\begin{proof}
Let $x_0\in\Gamma(u)$ and choose a supporting slope $\xi$. Then
\begin{equation}\label{lemma_UCS_def}
u(y)-u(x_0)-\langle\xi,y-x_0\rangle\leq0,\qquad y\in\Omega.
\end{equation}
For every $\eta\in\mathbb R^n$, the function $t\mapsto u(x_0+t\eta)-u(x_0)-t\langle\xi,\eta\rangle$ has a local maximum at zero. Its second derivative there is $\langle D^2u(x_0)\eta,\eta\rangle\leq0$. This proves the assertion.
\end{proof}

Next, for a given function $u \in C(\Omega)$ and  $r>0$, we denote
\begin{equation}\label{upper_contact_r}
\Gamma_r(u)=\left\{x \in \Omega :
\begin{array}{l}
\text{there exists } \xi=\xi(x)\in B_r(0) \text{ such that} \\
u(y) \le u(x) + \langle \xi, y-x \rangle\quad \text{for all } y \in \Omega
\end{array}
\right\}.
\end{equation}
Clearly, $\Gamma_r(u)\subseteq \Gamma(u)$ for all $r>0$ and $u \in C(\Omega)$. We also denote
\begin{equation}\label{def-M}
M(u)=\frac{\displaystyle{\sup_{x\in\Omega}u^+(x)-\sup_{x\in\partial\Omega}u^+(x)}}{\textup{diam}(\Omega)},
\end{equation}
where $u^+(x) = \max\{u(x), 0\}$ for $x \in \Omega$. We now prove $\Gamma_r (u)$ is contained in a compact subset of $\Omega$.

\begin{lemma} \label{Lemma-UCS_compact} 
Suppose that $\Omega \subset \R^n$ is a non-empty open bounded set and $u \in C(\overline{\Omega})$ satisfies $M(u)>0$. Then, for every  $r \in (0, M(u))$, there is $\delta = \delta(r, u)>0$ such that 
\begin{equation*}
\Gamma_r(u)\subseteq \Omega_\delta \quad \text{for} \quad \Omega_\delta = \{x\in\Omega: \textup{dist}(x,\partial\Omega)\ge \delta\}.
\end{equation*}
\end{lemma}

\begin{proof}
Let $R_0=\operatorname{diam}(\Omega)$ and $c=\sup_{\partial\Omega}u^+$. Since $M(u)>0$, there is $z\in\Omega$ such that $u(z)=\sup_\Omega u^+=c+R_0M(u)$. If $x\in\Gamma_r(u)$ has supporting slope $\xi\in B_r$, then
\[
u(z)\leq u(x)+\langle\xi,z-x\rangle\leq u(x)+rR_0.
\]
Consequently,
\begin{equation}\label{x-zero-y}
u(x)\geq c+R_0(M(u)-r).
\end{equation}
By uniform continuity on $\overline\Omega$, there is $\delta>0$ such that $|u(x)-u(y)|<R_0(M(u)-r)$ whenever $x,y\in\overline\Omega$ and $|x-y|<\delta$. Since $u\leq c$ on $\partial\Omega$, \eqref{x-zero-y} implies $\operatorname{dist}(x,\partial\Omega)\geq\delta$. The set $\Omega_\delta$ is compact and contained in $\Omega$.
\end{proof}

From the definition of $M(u)$ in \eqref{def-M}, we prove another essential lemma for our estimates.
\begin{lemma}\label{UCS lemma}
For $u\in C(\overline{\Omega})\cap C^1(\Omega)$, if $M = M(u)>0$, then for $0<r<M(u)$,
\begin{equation*}
B_r(0)= D u(\Gamma_r(u)).
\end{equation*}
\end{lemma}
\begin{proof}
Let us denote $R_0 = \operatorname{diam}(\Omega)$ and $c=\sup_{\partial\Omega}u^+$. We choose $x_0\in\Omega$ with $u(x_0)=c+R_0M(u)$. Fix $\xi\in B_r$ and let $z\in\overline\Omega$ maximize $u(y)-\langle\xi,y\rangle$, 
\begin{equation}\label{d-def}
h=u(z)-\langle\xi,z\rangle,\quad \text{and}\quad I_\xi(y)=h+\langle\xi,y\rangle.
\end{equation}
Then
\begin{equation}\label{u-I-map}
u(y)\leq I_\xi(y)\quad\text{in }\overline\Omega,\qquad u(z)=I_\xi(z).
\end{equation}
If $z\in\partial\Omega$, comparison with $x_0$ gives
\[
c+R_0M(u)=u(x_0)\leq u(z)+\langle\xi,x_0-z\rangle\leq c+rR_0,
\]
contrary to $r<M(u)$. Therefore $z\in\Omega$, $z\in\Gamma_r(u)$, and $Du(z)=\xi$. This proves $B_r\subset Du(\Gamma_r(u))$. Conversely, as $u \in C^1$, every supporting slope at a point equals the gradient of $u$ at that point. Hence, $Du(\Gamma_r(u))\subset B_r$ and therefore $B_r = Du(\Gamma_r(u))$. \end{proof}
\subsection{ABP Estimates for Classical Solutions} We give the proof of Theorem~\ref{ABP thm}.
\begin{proof}[Proof of Theorem \ref{ABP thm}]
If $M(u)\leq\gamma_0$, there is nothing to prove. Suppose $M(u)>\gamma_0$ and fix $r\in(\gamma_0,M(u))$. From Lemmas \ref{UCS D2u}-\ref{UCS lemma}, $\Gamma_r(u)$ is contained in a compact subset of $\Omega$, and $D^2u\leq0$ on $\Gamma_r(u)$, and $B_r=Du(\Gamma_r(u))$. Define
\[
E_r=\{x\in\Gamma_r(u):|Du(x)|\geq\gamma_0\},
\]
and recall that $g=g_1g_2$. The contact set $\Gamma(u)$ is relatively closed in $\Omega$, since for $u\in C^1(\Omega)$ it is given by the inequalities
\[
u(y)\leq u(x)+\langle Du(x),y-x\rangle,\qquad y\in\Omega.
\]
Also $\Gamma_r(u)=\Gamma(u)\cap\{|Du|<r\}$, so $E_r$ is measurable. The gradient is locally Lipschitz. Applying the weighted area formula, with multiplicities and the convention for measurable weights explained in Lemma \ref{area lemma}, yields
\begin{equation}\label{inte-1}
\int_{B_r}g(\xi)^{p_0}\,d\xi
\leq |B_{\gamma_0}|\|g\|_\infty^{p_0}
+\int_{E_r}g(Du)^{p_0}\det(-D^2u)\,dx.
\end{equation}
Here, we used $g_2(Du)\leq\|g_2\|_\infty$.  Moreover, we may discard the points of $E_r$ where $\det(-D^2u)=0$: the area formula applied to the preimage of any null set of slopes gives zero integral of $|\det D^2u|$ there. By the {\bf Assumption} \ref{F-assumption}, the equation, and $0\leq F\leq f^++|G|$, we have
\[
\begin{split}
g(Du)^{p_0}\det(-D^2u)
&\leq\mu g_2(Du)^{p_0}(f^++|G|)^{p_0}\\
&\leq2^{p_0-1}\mu\left(\|g_2\|_\infty^{p_0}(f^+)^{p_0}+b^{p_0}\right)
\end{split}
\]
almost everywhere on $E_r$. Then, it follows from this and \eqref{inte-1} that
\[
\int_{B_r}g^{p_0}\leq |B_{\gamma_0}|\|g\|_\infty^{p_0}
+2^{p_0-1}\left(\|g_2\|_\infty^{p_0}\|f^+\|_{L^{p_0}(\Omega,\mu)}^{p_0}
+\|b\|_{L^{p_0}(\Omega,\mu)}^{p_0}\right).
\]
From \eqref{gamma-def}, we note that $\|g\|_{L^{p_0}(B_r)}\leq\gamma$. Since $H$ is increasing, it follows that
\begin{equation}\label{r-H}
r\leq H(\gamma),\qquad \gamma_0<r<M(u).
\end{equation}
Letting $r\uparrow M(u)$, we obtain the required bound. 
\end{proof}

\section{Viscosity Solutions and Proof of Theorem \ref{G-ABP-thm}}

We first recall the concave-envelope construction used in ABP estimates, see \cite{Caffarelli,CC-1}. We then prove Theorem \ref{G-ABP-thm} directly on the upper contact set, using the locally Lipschitz gradient of the sub-solution. \subsection{Concave Envelopes} We recall the concept of concave envelopes and prove a few important properties.
\begin{definition}
\label{truncated envelope}
Let $u\in C(\overline{\Omega})$ and $r>0$. Define
\[
\mathcal{L}_r(u)
=\left\{ l(x)=a+\langle \xi,x\rangle: \ a \in \mathbb{R}, \ \xi \in B_r, \ 
l\geq u\ \text{in }\Omega
\right\}.
\]
The infimum below defines a finite concave, $r$-Lipschitz function on $\mathbb R^n$. Its restriction to $\Omega$, called the truncated concave envelope of $u$, is
\[
\Gamma_r(x)
=
\inf_{l\in\mathcal{L}_r(u)}l(x),
\qquad x\in\Omega .
\]
The corresponding contact set is
\[
A_r(u) =
\{x\in\Omega:\Gamma_r(x)=u(x)\}.
\]
\end{definition}

To verify finiteness, fix $y_0\in\Omega$. Every $l\in\mathcal L_r(u)$ satisfies
\[
l(x)\geq u(y_0)-r|x-y_0|,\qquad x\in\mathbb R^n,
\]
while the constant $\max_{\overline\Omega}u$ belongs to $\mathcal L_r(u)$. The infimum is concave because it is an infimum of affine functions. Since all their slopes have norm at most $r$, the infimum is also $r$-Lipschitz.

\smallskip
We note that $\Gamma_r$ depends on $u$ but we suppress this dependence for notational simplicity. Let us also recall the upper differentiability sets of concave functions.
\begin{definition}\label{definition upper differential}
Let $\Gamma:\Omega\rightarrow\mathbb{R}$ be a concave function. 
For $x\in\Omega$, the upper differential (or superdifferential) of $\Gamma$
at $x$ is defined by
\[
D^+\Gamma(x)
=
\left\{
\xi\in\mathbb{R}^n:
\Gamma(y)\leq
\Gamma(x)+\langle \xi,y-x\rangle
\quad\text{for all }y\in\Omega
\right\}.
\]
Equivalently, $\xi\in D^+\Gamma(x)$ if and only if the affine function
\[
y\mapsto \Gamma(x)+\langle \xi,y-x\rangle
\]
is a supporting hyperplane of $\Gamma$ from above at $x$. Similarly, for $r >0$, we also define
\[
D^+\Gamma_r(x)
=
\left\{
\xi\in\mathbb{R}^n:
\Gamma_r(y)\leq
\Gamma_r(x)+\langle\xi,y-x\rangle
\quad\forall y\in\Omega
\right\}.
\]
\end{definition}

Next, we state a few results needed to prove Theorem \ref{G-ABP-thm}. To keep the paper self-contained, we present all details in the proofs of needed lemmas. 

\begin{lemma}
\label{viscosity contact lemma}
Let $u\in C(\overline{\Omega})$ and assume
\[
M(u)=\frac{\sup_{\Omega}u^+-\sup_{\partial\Omega}u^+}{\textup{diam}(\Omega)}>0.
\]
Then, for every $0<r<M(u)$, we have
\begin{equation} \label{set-0718}
B_r(0) \subset \bigcup_{x\in A_r(u)}D^+\Gamma_r(x).
\end{equation}
Moreover, if
\[
E_r=
\left\{
x\in A_r(u):
\Gamma_r \text{ is twice differentiable at }x
\right\},
\]
then
\begin{equation} \label{E_r-0718}
B_r(0)\subset D\Gamma_r(E_r) \cup \big[ \displaystyle{\cup_{x\in (A_r(u)\setminus E_r)}}D^+\Gamma_r(x) \big].
\end{equation}
\end{lemma}
\begin{proof} We prove \eqref{set-0718}. Let $\xi\in B_r(0)$ be fixed and we consider the family of affine functions
\[
l_c(x)=\langle\xi,x\rangle+c, \quad c \in \mathbb{R}.
\]
Since $\Omega$ is bounded and $u \in C(\overline{\Omega})$, there exists $c$ such that
\[
l_c\geq u\quad\text{in }\Omega .
\]
Let $c_0$ be the smallest such value:
\[
c_0 = \inf\{ c \in \mathbb{R}: l_c \geq u \quad \text{in} \quad \Omega\}.
\]

The value $c_0$ equals $\max_{\overline\Omega}(u(x)-\langle\xi,x\rangle)$. A maximizer $x_0$ is interior by the boundary comparison in the proof of Lemma \ref{UCS lemma}, which uses only continuity. Thus $l_{c_0}(x_0)=u(x_0)$.
Since $|\xi|<r$, we have $l_{c_0}\in\mathcal{L}_r(u)$. Hence, it follows from the definition of $\Gamma_r$ that
\[
\Gamma_r(x) \leq l_{c_0}(x), \quad \forall \ x \in \Omega.
\]
Hence, $u(x_0)\leq\Gamma_r(x_0)\leq l_{c_0}(x_0)=u(x_0)$, which implies
\[
\Gamma_r(x_0)=u(x_0)\quad \text{and}\quad  x_0\in A_r(u).
\]
Since the plane $l_{c_0}$ supports $\Gamma_r$ from above at $x_0$, it follows that $\xi\in D^+\Gamma_r(x_0)$. Thus,
\[
B_r(0)\subset
\bigcup_{x\in A_r(u)}D^+\Gamma_r(x).
\]

\smallskip
It remains to prove \eqref{E_r-0718}. Since $\Gamma_r$ is concave, by Aleksandrov's theorem (see \cite[Theorem 6.9]{Evans}, for example), we see that it is twice
differentiable almost everywhere. At those points, we have
\[
D^+\Gamma_r(x)=\{D\Gamma_r(x)\}.
\]
Hence
\[
\bigcup_{x\in E_r}D^+\Gamma_r(x)  = D\Gamma_r(E_r).
\]
\end{proof}

Now for some open domain $\Omega \subset \R^n$, it follows from Rademacher's theorem (see \cite[Theorem 3.2, p. 103]{Evans}) that if $w \in C^{1,1}_{\text{loc}}(\Omega)$, then $D^2 w(x)$ exists almost everywhere in $\Omega$. Our next lemma records the area estimate in the regular case needed below. Without the stated Lipschitz condition on the gradient, the gradient image of the singular set of a concave function may have positive measure.
\begin{lemma}\label{area lemma}
Let $w\in C^{1,1}_{\mathrm{loc}}(\Omega)$ and let $E\Subset\Omega$ be measurable. For every nonnegative Lebesgue measurable function $h:\mathbb R^n\to[0,\infty]$, then it holds that
\[
\int_{Dw(E)}h(\xi)\,d\xi
\leq\int_E h(Dw(x))|\det D^2w(x)|\,dx.
\]
\end{lemma}
\begin{proof}
Put $T=Dw$ and $J_T=|\det DT|$, with $J_T=0$ where $T$ is not differentiable. Note that the map $T$ is locally Lipschitz, and therefore $DT$ exists almost everywhere in $\Omega$ due to Rademacher's theorem. Partition $E$ into disjoint measurable subsets contained in neighborhoods on which $T$ is Lipschitz. For nonnegative Borel functions $h$, the area formula on these subsets, followed by summation, gives
\[
\int_E h(T(x))J_T(x)\,dx
=\int_{\mathbb R^n}h(\xi)N(T,E,\xi)\,d\xi,
\]
where $N(T,E,\xi)$ is the number of preimages of $\xi$ in $E$; see \cite[Chapter 3]{Evans}. Since $N(T,E,\xi)\geq1$ on $T(E)$, the asserted inequality follows in this case. Locally Lipschitz maps send null sets to null sets, so the points where $T$ is not differentiable do not affect the left-hand integral.

For a general Lebesgue measurable $h$, choose a nonnegative Borel representative $\widetilde h$ equal to $h$ outside a Borel null set $Z$. The area formula applied to $E\cap T^{-1}(Z)$ gives
\[
\int_{E\cap T^{-1}(Z)}J_T(x)\,dx=0.
\]
Thus $E\cap T^{-1}(Z)\cap\{J_T>0\}$ has measure zero. With the stated convention on $\{J_T=0\}$, the product $h(T)J_T$ is measurable and agrees almost everywhere with $\widetilde h(T)J_T$. This proves the assertion. The same argument justifies the use of essential suprema of the gradient weights on the set where the contact Hessian has positive determinant.
\end{proof}
The following lemma explains how the viscosity inequality can be evaluated at an Aleksandrov contact point.
\begin{lemma} \label{Gamma viscosity test lemma}
Let $u\in C(\overline{\Omega})$ be a viscosity sub-solution of \eqref{main eqn} for some given $f$, where $F$ and $G$ are continuous. Let $\Gamma_r$ be the truncated concave envelope of $u$ defined in Definition \ref{truncated envelope}. Then
\[
F(x_0,u(x_0),D\Gamma_r(x_0),D^2\Gamma_r(x_0))
+
G(x_0,u(x_0),D\Gamma_r(x_0))
\leq f(x_0)
\]
for every Aleksandrov point $x_0\in A_r(u)$ of $\Gamma_r$.
\end{lemma}
\begin{proof}
Since $\Gamma_r$ is twice differentiable at $x_0$ in the
Aleksandrov sense (see \cite[Theorem 6.9]{Evans}), there are $D\Gamma_r(x_0) \in \mathbb{R}^n$ and $D^2\Gamma_r(x_0) \in S(n)$ such that
\begin{align*}
\Gamma_r(x) & =
\Gamma_r(x_0) +
\langle D\Gamma_r(x_0),x-x_0\rangle +
\frac12
\langle D^2\Gamma_r(x_0)(x-x_0),x-x_0\rangle
\\
& \qquad +
o(|x-x_0|^2), \quad \text{for} \ x \ \text{sufficiently close to }  x_0.
\end{align*}
Therefore, for every $\varepsilon>0$, the function
\[
\psi_\varepsilon(x)
=
\Gamma_r(x_0)
+
\langle D\Gamma_r(x_0),x-x_0\rangle
+
\frac12
\langle
(D^2\Gamma_r(x_0)+\varepsilon I_n)(x-x_0),
x-x_0
\rangle
\]
lies above $\Gamma_r$ in a sufficiently small neighborhood of $x_0$. Since $\Gamma_r\geq u$ and
\[
\Gamma_r(x_0)=u(x_0),
\]
we have
\[
\psi_\varepsilon\geq u \quad \text{
near $x_0$ and} \quad  \psi_\varepsilon(x_0)=u(x_0).
\]
 Hence $\psi_\varepsilon$ is an
admissible test function from above for the sub-solution $u$. It follows that
\[
F(x_0,u(x_0),D\psi_\varepsilon(x_0),
D^2\psi_\varepsilon(x_0)) + G(x_0,u(x_0),D\psi_\varepsilon(x_0))
\leq f(x_0).
\]
Moreover,
\[
D\psi_\varepsilon(x_0)=D\Gamma_r(x_0), \quad \text{and} \quad
D^2\psi_\varepsilon(x_0)
=
D^2\Gamma_r(x_0)+\varepsilon I .
\]
Therefore,
\[
F(x_0,u(x_0),D\Gamma_r(x_0),
D^2\Gamma_r(x_0)+\varepsilon I_n) + G(x_0,u(x_0),D\Gamma_r(x_0))
\leq f(x_0).
\]
Letting $\varepsilon\to0$ and using the continuity of $F$ and $G$ gives
\[
F(x_0,u(x_0),D\Gamma_r(x_0),D^2\Gamma_r(x_0))
+
G(x_0,u(x_0),D\Gamma_r(x_0))
\leq f(x_0).
\]
The proof is completed.
\end{proof}

\subsection{ABP Estimate for Viscosity Solutions}
We now prove Theorem \ref{G-ABP-thm}. The upper contact set from the classical argument in the proof of Theorem \ref{ABP thm} can be used for a $C^{1,1}_{\mathrm{loc}}$ viscosity sub-solution. At almost every point of that set, its quadratic Taylor polynomial with a positive quadratic perturbation is an admissible viscosity test function.
\begin{proof}[Proof of Theorem \ref{G-ABP-thm}]
Let $c=\sup_{\partial\Omega}u^+$, $v=u-c$, and set
\[
M=\frac{\sup_\Omega u^+-c}{\operatorname{diam}(\Omega)}.
\]
The estimate is immediate when $M\leq\gamma_0$. Suppose $M>\gamma_0$ and fix $r\in(\gamma_0,M)$. Let $C_r$ be the set of all $x\in\overline\Omega$ such that an affine function of slope $\xi\in\overline B_r$ touches $v$ from above at $x$ and lies above $v$ throughout $\overline\Omega$. Passing to limits in the supporting inequalities shows that $C_r$ is closed in $\overline\Omega$, by compactness of $\overline B_r$. Its contact points lie in a compact subset of $\Omega$: if $x\in C_r$, comparison with a maximum point of $v$ gives $v(x)\geq(M-r)\operatorname{diam}(\Omega)>0$, and uniform continuity separates such points from the boundary.

\smallskip
For every $\xi\in B_r$, the continuous function $v(y)-\langle\xi, y\rangle$ attains its maximum at an interior point $x$. Indeed, if a boundary point $y$ maximized it, comparison with a point $z$ where $v$ attains its positive maximum would give
\[
M\operatorname{diam}(\Omega)\leq v(z)\leq\langle\xi,(z-y)\rangle\leq r\operatorname{diam}(\Omega),
\]
a contradiction. At that interior point, $Dv(x)=\xi$ and $x\in C_r$. Thus $B_r\subset Dv(C_r)$. Moreover, wherever $D^2v$ exists on $C_r$, the supporting affine function implies $D^2v\leq0$.

Split $C_r$ into $C_r^1=C_r\cap\{|Dv|<\gamma_0\}$ and $C_r^2=C_r\cap\{|Dv|\geq\gamma_0\}$. By Lemma \ref{area lemma} and the inclusion just proved,
\begin{equation}\label{vis-area-corrected}
\int_{B_r}g(\xi)^{p_0}\,d\xi\leq
\int_{B_{\gamma_0}}g(\xi)^{p_0}\,d\xi
+\int_{C_r^2}g(Dv)^{p_0}\det(-D^2v)\,dx.
\end{equation}
Here and below, $g=g_1g_2$ as defined in \eqref{H-def}. We apply the viscosity inequality at almost every point $x\in C_r^2$. By Rademacher's theorem, $Du$ is differentiable almost everywhere. Its derivative agrees almost everywhere with the weak Hessian and is symmetric. At each such point $x$, writing $A=D(Du)(x)$ and integrating along the segment from $x$ to $x+h$ gives
\[
\begin{split}
u(x+h)-u(x)-\langle Du(x), h\rangle
&=\int_0^1\left\langle\big(Du(x+th)-Du(x)\big), h\right\rangle\,dt\\
&=\tfrac12 \langle h ,Ah\rangle+o(|h|^2).
\end{split}
\]
Consequently $D^2u(x)=A$ in the second-order expansion. For every $\varepsilon>0$, the polynomial
\[
P_\varepsilon(y)=u(x)+\langle Du(x),(y-x)\rangle
+\tfrac12(y-x)\cdot A(y-x)+\varepsilon|y-x|^2
\]
touches $u$ from above in a sufficiently small neighborhood of $x$. Letting $\varepsilon\downarrow0$ and using continuity of $F$ and $G$ yields
\[
F(x,u(x),Du(x),D^2u(x))+G(x,u(x),Du(x))\leq f(x).
\]
At almost every such point with $\det(-D^2u)>0$, the gradient also avoids the null set of slopes where $g_2$ exceeds its essential supremum. By Assumption \ref{F-assumption}, $F\geq0$ on this set and
\[
\begin{split}
g(Du)^{p_0}\det(-D^2u)
&\leq\mu(x)g_2(Du)^{p_0}F(x,u,Du,D^2u)^{p_0}\\
&\leq2^{p_0-1}\mu(x)\left(\|g_2\|_\infty^{p_0}(f^+)^{p_0}+b(x)^{p_0}\right).
\end{split}
\]
The same bound holds trivially where the determinant vanishes. Inserting this estimate into \eqref{vis-area-corrected} and using \eqref{gamma-def} gives
\[
\|g\|_{L^{p_0}(B_r)}\leq\gamma.
\]
Let $r\uparrow M$. If $\gamma=0$, the inequality for any $r>0$ already gives a contradiction. Otherwise, since $\gamma<\|g\|_{L^{p_0}(\mathbb R^n)}$ and the radial integral of $g^{p_0}$ is strictly increasing, $M\leq H(\gamma)$. Combining this with the case $M\leq\gamma_0$ gives \eqref{g-abp}.
\end{proof}

\section{Applications} \label{mean-curve}
We apply Theorem \ref{G-ABP-thm} and its contact-set argument to a class of generalized Monge-Amp\`ere equations, Hessian equations, Hessian quotients under an additional condition on the contact Hessian, and anisotropic mean-curvature-type equations. For the mean curvature applications, we treat the bounded normalized drift and the normalized quadratic gradient term separately. In this section, we always assume $\Omega \subset \mathbb{R}^n$ is open and bounded.

\subsection{ABP Estimates for  Generalized Monge-Amp\`ere Equations} \label{Monge-sec} 
An application of Theorem \ref{G-ABP-thm} is for the class of generalized Monge-Amp\`ere equations that are nonlinear in $(x,u,Du)$. We formulate the equation of interest as \eqref{MA eqn}.
\begin{equation}\label{MA eqn}
\Phi(x,u(x),Du(x))\textsf{det}(-D^2u(x))+G(x,u(x),Du(x)) = f(x)\quad\text{on}\ \Omega
\end{equation}
where $\Omega \subset\mathbb{R}^n$ is an open bounded domain with boundary $\partial \Omega$, $f : \Omega \rightarrow \mathbb{R}$ is a given function, and $\Phi:\Omega\times\mathbb R\times\mathbb R^n\to(0,\infty)$ and $G:\Omega\times\mathbb R\times\mathbb R^n\to\mathbb R$ are functions satisfying the following conditions. We would like to point out that related Monge-Amp\`ere type equations appear in the optical design for illumination applications, optimal transportation, geometric optics, and conformal geometry (e.g. \cite{JTY, GN, Prins}).

\smallskip
We consider concave sub-solutions. For viscosity tests, the equation is understood using the continuous extension
\begin{equation}\label{MA-extension}
\widehat F(x,z,\xi,X)=\Phi(x,z,\xi)\det((-X)_+),\qquad X\in S(n).
\end{equation}
Thus a test function touching $u$ from above is evaluated by \eqref{MA-extension}. On non-positive Hessians, this agrees with \eqref{MA eqn}. The extension is non-increasing in $X$ by the monotonicity of the ordered eigenvalues. The unrestricted signed polynomial $\det(-X)$ is not used for upper tests with indefinite or positive Hessians.

\begin{assumption}\label{G-assumption} We assume that the following conditions hold.
\begin{enumerate}[label=\normalfont(\arabic*)]
\item\label{G-assumption-1} $\Phi, G \in C(\overline{\Omega}\times\mathbb R\times \mathbb R^n)$ and $\Phi>0$.
\item\label{G-assumption-2} There exists $\mu:\mathbb R^n\to (0, \infty)$ locally integrable such that 
\[
1\le \mu(x)\Phi(x,z,\xi), \qquad (x,z,\xi)\in\Omega\times\mathbb R\times(\mathbb R^n\setminus B_{\gamma_0})
\] 
for some $\gamma_0>0$.
\item\label{G-assumption-3} There exist a nonnegative function $b\in L^1(\Omega, \mu)$ and a constant $m\in [0,n)$ such that
\[ 
|G(x,z,\xi)| \leq b(x)|\xi|^m, \quad (x,z,\xi)\in\Omega\times\mathbb R\times (\mathbb R^n \setminus B_{\gamma_0}).
\]
\end{enumerate}
\end{assumption}

We now state and prove the main theorem of this subsection.

\begin{theorem}\label{monge-ampere}
Let $\Phi,\ G:\Omega\times\mathbb R\times\mathbb R^n\to \mathbb R$ be measurable functions satisfying \textup{Assumption \ref{G-assumption}}. Let $f\in C(\Omega)$ satisfy $f^+\in L^1(\Omega,\mu)$. Then, for any concave viscosity sub-solution $u\in C(\overline\Omega)\cap C^{1,1}_{\mathrm{loc}}(\Omega)$ to equation \eqref{MA eqn}, we have
\begin{equation}\label{MA-estimate}
\sup_\Omega u^+\le \sup_{\partial\Omega}u^+ +\textup{diam}(\Omega)\max\big\{\gamma_0, I\left(1+\|b\|_{L^1(\Omega,\mu)}\right)\big\},
\end{equation}
where $b(x)$ is the function given in Assumption \ref{G-assumption}, and $I(s)$ is defined as
\begin{equation}\label{function I}
I(s)=
\begin{cases}
\displaystyle
\left(\frac{s}{s_0}\right)^{\frac{1}{n}}, & 0\le s\leq s_0,\\[1.2em]
\displaystyle
\left[1+\left(1-\frac{m}{n}\right)\left(\frac{s}{s_0}-1\right)\right]^{1/(n-m)}, & s> s_0,
\end{cases}
\end{equation}
with $s_0=\omega_n(1+\|f^+\|_{L^1(\Omega,\mu)})^{-1}>0$, and $\omega_n$ being the Lebesgue measure of the unit ball in $\mathbb{R}^n$.
\end{theorem}

\begin{proof}
Let $a=\|f^+\|_{L^1(\Omega,\mu)}$ and $B=\|b\|_{L^1(\Omega,\mu)}$. We first suppose $a>0$ and take $\gamma_0=0$, $p_0=1$, and
\[
g_1(\xi)=1,\qquad g_2(\xi)=\frac{1}{|\xi|^m+a}.
\]
The functions $g_1,g_2$ are bounded and positive. For $X\leq0$, Assumption \ref{G-assumption}-\ref{G-assumption-2} gives
\[
\det(-X)\leq\mu(x)\Phi(x,z,\xi)\det((-X)_+)
=\mu(x)\widehat F(x,z,\xi,X).
\]
Also,
\[
g_2(\xi)|G(x,z,\xi)|
\leq\frac{|\xi|^m}{|\xi|^m+a}b(x)\leq b(x).
\]
Thus Assumption \ref{F-assumption} holds. By \eqref{gamma-def},
\begin{equation}\label{gamma3}
\gamma=a\|g_2\|_\infty+B\leq1+B.
\end{equation}
For $m>0$ there is equality, whereas for $m=0$ the first term equals $a/(1+a)$.

By polar coordinates,
\[
Q_a(r)=\|g_1g_2\|_{L^1(B_r)}
=n\omega_n\int_0^r\frac{\rho^{n-1}}{\rho^m+a}\,d\rho.
\]
Since $m < n$, we have
\begin{equation}\label{H-infinity}
\lim_{r\rightarrow \infty}Q_a(r)=\|g_1g_2\|_{L^1(\mathbb R^n)}=\infty.
\end{equation}
Let 
\[
L_a(r) = \left\{
\begin{array}{ll}
\frac{\omega_n}{1+a}r^n & \quad \text{if} \quad r \in [0, 1]\\
\frac{\omega_n}{1+a}
+\frac{n\omega_n}{(1+a)(n-m)}(r^{n-m}-1) & \quad \text{if} \quad r >1.
\end{array} \right.
\]
For $0\leq r\leq1$, $\rho^m+a\leq1+a$ on the interval of integration, and hence
\[
Q_a(r)\geq\frac{\omega_n}{1+a}r^n.
\]
For $r\geq1$, use $\rho^m+a\leq(1+a)\rho^m$ on $[1,r]$ to obtain
\[
Q_a(r)\geq\frac{\omega_n}{1+a}
+\frac{n\omega_n}{(1+a)(n-m)}(r^{n-m}-1).
\]
Hence, we have
\[
Q_a(r) \geq L_a(r), \quad \forall \ r \geq 0.
\]
We also note that $L_a(r)$  is continuous and strictly increasing from zero to infinity. Its inverse is exactly the function $I$ in \eqref{function I}, with $s_0=\omega_n/(1+a)$. Since $Q_a\geq L_a$, their inverses satisfy
\begin{equation}\label{H-I}
H(s)=Q_a^{-1}(s)\leq L_a^{-1}(s)=I(s),\qquad s\geq0.
\end{equation}
Theorem \ref{G-ABP-thm}, \eqref{gamma3}, and \eqref{H-I} now give
\[
\begin{split}
\sup_\Omega u^+
&\leq\sup_{\partial\Omega}u^++\operatorname{diam}(\Omega)\max\{\gamma_0, H(\gamma)\}\\
&\leq\sup_{\partial\Omega}u^++\operatorname{diam}(\Omega)\max\{ \gamma_0, I(1+B)\}.
\end{split}
\]
If $a=0$, apply this argument with $f_\varepsilon=f^++\varepsilon$. The function $u$ remains a sub-solution, and $a_\varepsilon=\|f_\varepsilon\|_{L^1(\Omega,\mu)}\to0$ because $\mu(\Omega)<\infty$. The explicit expression \eqref{function I} is continuous in $a\geq0$, including where its two branches meet. Passing to the limit gives the stated estimate when $a=0$.
\end{proof}

\begin{remark}
When $G=0$, the direct choice $g_1=g_2=1$ gives the sharper bound
\[
\sup_\Omega u^+\leq\sup_{\partial\Omega}u^+
+\operatorname{diam}(\Omega)\left(\frac{\|f^+\|_{L^1(\Omega,\mu)}}{\omega_n}\right)^{\frac{1}{n}}.
\]
For $\Phi=1$ and $\mu=1$, this is the classical Alexandrov estimate for the Monge-Amp\`ere equation; see \cite{GT, HP}. In particular, the estimate vanishes with the positive part of the source. The expression \eqref{MA-estimate} is a convenient uniform bound in the presence of the gradient term, and is not intended to give this sharper zero-drift dependence.
\end{remark}

\begin{remark}\label{critical-MA}
The same argument also treats growth of order $n$. More precisely, retain the other hypotheses of Theorem \ref{monge-ampere}, and replace the last condition in Assumption \ref{G-assumption} by
\[
|G(x,z,\xi)|\leq b(x)|\xi|^n,\qquad b\geq0,\quad b\in L^1(\Omega,\mu).
\]
Put $a=\|f^+\|_{L^1(\Omega,\mu)}$ and $B=\|b\|_{L^1(\Omega,\mu)}$. Then
\[
\sup_\Omega u^+\leq\sup_{\partial\Omega}u^+
+\operatorname{diam}(\Omega)\left[a\left(e^{(1+B)/\omega_n}-1\right)\right]^{\frac{1}{n}}.
\]
Indeed, for $a>0$, take $p_0=1$, $\gamma_0 \rightarrow 0^+$, $g_1=1$, and $g_2(\xi)=(|\xi|^n+a)^{-1}$. The structural assumptions hold, $\gamma=1+B$, and
\[
\|g_1g_2\|_{L^1(B_r)}
=n\omega_n\int_0^r\frac{t^{n-1}}{t^n+a}\,dt
=\omega_n\log(1+r^n/a).
\]
This integral tends to infinity, and its inverse gives the asserted estimate by Theorem \ref{G-ABP-thm}. If $a=0$, apply the estimate with source $f^++\varepsilon$ and let $\varepsilon\downarrow0$. Since $\mu(\Omega)<\infty$, the resulting bound converges to the maximum principle.
\end{remark}
\subsection{ABP Estimates for \texorpdfstring{$k$}{k}-Hessian Equations with Gradient Lower-order Terms} \label{k-Hessian}
To begin with, let us recall the definition of $k$-Hessian operators. Let $X\in S(n)$ with eigenvalues
\[
\lambda_1(X) \leq \lambda_2(X) \leq \ldots \leq \lambda_n(X).
\]
For $k=1,2, \ldots,n$, the $k$-th elementary symmetric function of eigenvalues of $X$ is defined by
\[
\sigma_k(X) =\sum_{1\leq i_1<\cdots<i_k\leq n}
\lambda_{i_1}(X)\cdots\lambda_{i_k}(X).
\]
For a sufficiently smooth function $u:\Omega\rightarrow\mathbb{R}$,
the $k$-Hessian of $u$ is defined as
\[
\sigma_k(D^2u) = S_k(\lambda(D^2u)) =  \sum_{1\leq i_1<\cdots<i_k\leq n}
\lambda_{i_1}(D^2u)\cdots\lambda_{i_k}(D^2u).
\]
where $\lambda(A)$ is the set of eigenvalues for any symmetric matrix $A$. In particular,
\[
\sigma_1(D^2u)=\operatorname{tr}(D^2u)=\Delta u,
\quad \text{and} \quad \sigma_n(D^2u)=\det(D^2u).
\]
The following lemma is known as Newton-Maclaurin inequality, see \cite[Theorem B.1.1]{NP}.
\begin{lemma}\label{sigma-det}
Let $A\in S(n)$ be a non-negative definite matrix. Then, for every
$k=1,\ldots,n$, it holds that
\begin{equation}\label{sigma-det-ineq}
\binom{n}{k}^{\frac{1}{k}}\det(A)^{\frac{1}{n}}
\leq \sigma_k(A)^{\frac{1}{k}},
\end{equation}
\end{lemma}
\begin{proof} Let $\lambda_1,\ldots,\lambda_n\geq 0$ be the eigenvalues of $A$. Then
\[
\sigma_k(A)
=
\sum_{1\leq i_1<\cdots<i_k\leq n}
\lambda_{i_1}\cdots\lambda_{i_k}.
\]
Note that there are exactly $N=\binom{n}{k}$ terms in the above sum. Applying the arithmetic--geometric mean
inequality to these non-negative terms gives
\[
\frac{\sigma_k(A)}{N}
\geq
\left(
\prod_{1\leq i_1<\cdots<i_k\leq n}
\lambda_{i_1}\cdots\lambda_{i_k}
\right)^{\frac1N}.
\]
Now, note that each eigenvalue $\lambda_i$ appears in exactly $\binom{n-1}{k-1}$ of the products $\lambda_{i_1}\cdots\lambda_{i_k}$  above because after fixing $\lambda_i$, the remaining $k-1$ indices can be
chosen from the other $n-1$ indices. Consequently,
\[
\prod_{1\leq i_1<\cdots<i_k\leq n}
\lambda_{i_1}\cdots\lambda_{i_k}
=
\prod_{i=1}^{n}
\lambda_i^{\binom{n-1}{k-1}}
=
\det(A)^{\binom{n-1}{k-1}}.
\]
Therefore,
\[
\sigma_k(A)
\geq
\binom{n}{k}
\det(A)^{
\frac{\binom{n-1}{k-1}}{\binom{n}{k}}
}
=
\binom{n}{k}
\det(A)^{\frac{k}{n}},
\]
since $\binom{n-1}{k-1}/\binom nk=k/n$. Taking the power $\frac1k$ on both sides gives
\[
\sigma_k(A)^{\frac1k}
\geq
\binom{n}{k}^{\frac1k}
\det(A)^{\frac1n},
\]
which is equivalent to \eqref{sigma-det-ineq}. The proof is completed.
\end{proof}
For a symmetric matrix $X$, let $X_+$ denote the positive spectral part, obtained by replacing each negative eigenvalue by zero. In \eqref{k-Hessian equation}, viscosity tests are evaluated using the continuous operator
\[
\widehat F_k(X)=\sigma_k((-X)_+)^{\frac{1}{k}}.
\]
It agrees with $\sigma_k(-X)^{\frac{1}{k}}$ for $X\leq0$ and is non-increasing with respect to the matrix order. This convention is needed because an upper test of a concave function need not itself have a non-positive Hessian.

\smallskip
We now state and prove the following result on the ABP estimate for $k$-Hessian equations with gradient lower-order terms.
\begin{theorem}
\label{k-Hessian ABP lemma}
Let $\Omega\subset\mathbb{R}^n$ be a bounded domain and let
$k\in\{1,\ldots,n\}$. Suppose that a concave function $u\in C(\overline{\Omega})\cap C^{1,1}_{\mathrm{loc}}(\Omega)$ is a viscosity sub-solution of
\begin{equation}\label{k-Hessian equation}
\sigma_k(-D^2u)^{\frac1k}
+
G(x,u,Du)
=
f(x)
\quad\text{in }\Omega ,
\end{equation}
where $G \in C(\Omega\times\mathbb{R}\times\mathbb{R}^n)$ and $f\in C(\Omega)$ with $f^+\in L^n(\Omega)$. Assume that there exists a bounded measurable function
$g_2:\mathbb{R}^n\rightarrow(0,\infty)$ and a function $b\in L^n(\Omega)$ such that
\begin{equation}\label{G-kHessian condition}
g_2(\xi)|G(x,z,\xi)|
\leq b(x),
\qquad
(x,z,\xi)\in
\Omega\times\mathbb{R}\times (\mathbb{R}^n \setminus B_{\gamma_0})
\end{equation}
with some $\gamma_0 >0$.
Let 
\[
 \gamma
= 2^{\frac{n-1}{n}}\Big(
\|g_2\|_{L^\infty(\mathbb{R}^n)}^n
\|f^+\|_{L^n(\Omega)}^n +
\|b\|_{L^n(\Omega)}^n
\Big)^{\frac{1}{n}},
\]
If
\[
\gamma
< N \|g_2\|_{L^n(\mathbb{R}^n)} \quad \text{for} \quad N = \binom{n}{k}^{\frac1k}
\]
then
\[
\sup_\Omega u^+
\leq
\sup_{\partial\Omega}u^+
+
\operatorname{diam}(\Omega) \max\{\gamma_0, H(\gamma)\},
\]
where $H:[0,N\|g_2\|_{L^n(\mathbb R^n)})\to[0,\infty)$ is the inverse defined by $H \big(N \|g_2\|_{L^n(B_r)} \big)=r$ for all $r>0$, with $H(0)=0$.
\end{theorem}
\begin{proof} We verify that the $k$-Hessian operator satisfies the structural assumption
of Theorem \ref{G-ABP-thm}. Define $F(X)=\sigma_k((-X)_+)^{\frac{1}{k}}$, where $A_+$ denotes the positive spectral part of a symmetric matrix $A$. This continuous extension agrees with $\sigma_k(-X)^{\frac{1}{k}}$ when $X\leq0$.
Let $\phi\in C^2_{\mathrm{loc}}(\Omega)$ be such that
\[
D^2\phi(x)\leq0.
\]
By the Newton-Maclaurin inequality (Lemma \ref{sigma-det}), it follows that 
\[
\sigma_k(-D^2\phi(x))^{\frac1k}
\geq
N \det(-D^2\phi(x))^{\frac1n}, \quad \text{for} \quad N = \binom{n}{k}^{\frac1k}.
\]
Therefore,
\[
N
\det(-D^2\phi(x))^{\frac1n}
\leq
F(D^2\phi(x)).
\]
Hence Assumption \ref{F-assumption}-\ref{F-assumption-2} is satisfied with
\[
p_0=n,\qquad
g_1(\xi)\equiv N,
\quad \text{and} \quad
\mu(x)\equiv1,
\]
and, since the inequality holds for every gradient $\xi\in\mathbb R^n$, $\gamma_0=0.$

\smallskip
Next, by the assumption on the lower order term, we have
\[
g_2(\xi)|G(x,z,\xi)|
\leq b(x), \quad \forall 
(x,z,\xi)\in\Omega\times\mathbb R\times (\mathbb R^n \setminus B_{\gamma_0}).
\]
Thus Assumption \ref{F-assumption}-\ref{F-assumption-3} is satisfied. Therefore, all hypotheses of Theorem \ref{G-ABP-thm} are fulfilled with
\[
p_0=n.
\]
Then, by applying Theorem \ref{G-ABP-thm}, we obtain
\[
\sup_\Omega u^+
\leq
\sup_{\partial\Omega}u^+
+
\operatorname{diam}(\Omega)
\max\{\gamma_0 ,H(\gamma)\}.
\]
Since $H(\gamma)\geq0$, we obtain
\[
\sup_\Omega u^+
\leq
\sup_{\partial\Omega}u^+
+
\operatorname{diam}(\Omega)\max\{ \gamma_0, H(\gamma) \},
\]
which proves the desired estimate.
\end{proof}

\noindent
When $G=0$, the contact-set argument in the preceding proof gives the sharper estimate
\[
\sup_\Omega u^+\leq\sup_{\partial\Omega}u^+
+\frac{\operatorname{diam}(\Omega)}{\binom nk^{\frac{1}{k}}\omega_n^{\frac{1}{n}}}
\|f^+\|_{L^n(\Omega)}
\]
where $\omega_n = |B_1|$ is the Lebesgue measure of the unit ball in $\mathbb R^n$. Indeed, one integrates $N^n\det(-D^2u)\leq(f^+)^n$ directly, without splitting the source and lower-order terms. This is the usual determinant-majorization estimate; see also \cite[Theorem 4.5]{HP} for polynomial operators and admissible viscosity solutions.

\begin{example}\label{exponential-Hessian}
The gradient weights also allow non-power dependence in a fully nonlinear equation. Let $n\geq2$, $2\leq k\leq n$, $\beta>0$, and let $a>0$ and $q\geq0$ be continuous in $\Omega$. Suppose $\mu=a^{-n}$ in $\Omega$ extends to a positive locally integrable function on $\mathbb R^n$, and $q,f^+\in L^n(\Omega,\mu)$, with $f$ continuous. Consider a viscosity sub-solution $u\in C(\overline\Omega)\cap C^{1,1}_{\mathrm{loc}}(\Omega)$ of
\[
a(x)(1+|Du|^2)^{-\beta/2}\sigma_k((-D^2u)_+)^{\frac{1}{k}}
-q(x)(e^{|Du|^2}-1)\leq f(x).
\]
Put $N=\binom nk^{\frac{1}{k}}$ and set
\[
\begin{split}
W(r)&=N^n n\omega_n\int_0^r t^{n-1}(1+t^2)^{-\beta n/2}e^{-nt^2}\,dt,\\
S&=2^{n-1}\left(\|f^+\|_{L^n(\Omega,\mu)}^n+\|q\|_{L^n(\Omega,\mu)}^n\right),
\end{split}
\]
where $\omega_n=|B_1|$ is the Lebesgue measure of the unit ball in $\mathbb R^n$. If $S<W(\infty)$, then
\[
\sup_\Omega u^+\leq\sup_{\partial\Omega}u^+
+\operatorname{diam}(\Omega)W^{-1}(S).
\]
To verify this, take $p_0=n$, $\gamma_0 \rightarrow 0^+$, $g_1(\xi)=N(1+|\xi|^2)^{-\beta/2}$, $g_2(\xi)=e^{-|\xi|^2}$, and $b=q$. The Newton--Maclaurin inequality and $\mu^{\frac{1}{n}}a=1$ give the determinant condition, while
\[
g_2(\xi)|G(x,z,\xi)|=q(x)(1-e^{-|\xi|^2})\leq q(x).
\]
Since $W(r)=\|g_1g_2\|_{L^n(B_r)}^n$ and $0<W(\infty)<\infty$, the conclusion follows from Theorem \ref{G-ABP-thm}. This example combines the Hessian dependence with variable coefficients and a nonlinear gradient term.
\end{example}

\subsection{ABP Estimates for Hessian quotient equations} In this subsection, we apply Theorem  \ref{G-ABP-thm} to Hessian quotient
equations. Throughout this section, for $0\leq l<k\leq n$, we consider
the operator
\begin{equation}\label{F-kl-def}
F_{k,l}(A)=\left(\frac{\sigma_k(A)}{\sigma_l(A)}\right)^{\frac1{k-l}},
\end{equation}
where $\sigma_j(A)$ denotes the $j$-th elementary symmetric function of the eigenvalues of $A\in S(n)$ defined in Section \ref{k-Hessian}.

\smallskip
We introduce the admissible cone
\[
\mathcal C_k=\left\{ \lambda\in\mathbb R^n:
\sigma_j(\lambda)>0,\quad j=1,\ldots,k \right\}.
\]
The operator $F_{k,l}$ is elliptic in $\mathcal C_k$. For a function $u\in C^2(\Omega)$, we say that $u$ is $k$-admissible (in the concave sense) if
\[
\lambda(-D^2u(x))\in \mathcal C_k \quad \text{for every } x \in \Omega.
\]
For $A\in S(n)$, define the normalized elementary symmetric functions
\[
E_j(A)=\frac{\sigma_j(A)}{\binom nj},
\qquad j=0,\ldots,n .
\]
Note that by convention $\sigma_0(A) =1$, hence
\[
E_0(A) =1.
\]
We first recall the Newton inequalities; see \cite[Theorem B.1.1]{NP}, for example.
\begin{lemma}
\label{Newton inequalities}
Let $A\geq0$. Then
\[
E_j(A)^2\geq E_{j-1}(A)E_{j+1}(A), \qquad j=1,\ldots,n-1 .
\]
\end{lemma}

\begin{lemma}\label{log-concavity}
Let $A\in S(n)$ be a positive definite matrix. Then the following assertions are true.
\begin{enumerate}[label=\rm(\roman*)]
\item The sequence $j\longmapsto \log E_j(A)$ is concave, that is,
\[
2\log E_j(A)
\ge
\log E_{j-1}(A)+\log E_{j+1}(A),
\qquad
j=1,\ldots,n-1.
\]
\item Let
\[
\Delta_j
=
\log E_j(A)-\log E_{j-1}(A),
\qquad
j=1,\ldots,n,
\]
then
\[
\Delta_1\ge\Delta_2\ge\cdots\ge\Delta_n.
\]
\item Consequently, for every $0\le l<k\le n$, we have
\[
\frac{\log E_k(A)-\log E_l(A)}{k-l}
\ge
\frac{\log E_n(A)-\log E_l(A)}{n-l}.
\]
\end{enumerate}
\end{lemma}

\begin{proof} We begin with the proof of (i). By Newton's inequality (Lemma \ref{Newton inequalities}), we have
\[
E_j(A)^2
\ge
E_{j-1}(A)E_{j+1}(A).
\]
It implies
\[
\log\big(E_j(A)^2\big)
\ge
\log\big(E_{j-1}(A)E_{j+1}(A)\big).
\]
Then, we obtain
\[
2\log E_j(A)
\ge
\log E_{j-1}(A)+\log E_{j+1}(A),
\]
which proves (i).

\smallskip 
We prove (ii). Rearranging the inequality in (i), we obtain
\[
\log E_j(A)-\log E_{j-1}(A)
\ge
\log E_{j+1}(A)-\log E_j(A).
\]
By the definition of $\Delta_j$, we see that
\[
\Delta_j\ge\Delta_{j+1},
\qquad
j=1,\ldots,n-1.
\]
Hence,
\[
\Delta_1\ge\Delta_2\ge\cdots\ge\Delta_n,
\]
and (ii) follows.

\smallskip It remains to prove (iii). Observe that for $0 \leq l < k \leq n$,
\[
\log E_k(A)-\log E_l(A)
=
\sum_{i=l+1}^{k}\Delta_i.
\]
Therefore,
\[
\frac{\log E_k(A)-\log E_l(A)}{k-l}
=
\frac{\Delta_{l+1}+\cdots+\Delta_k}{k-l},
\]
which is the arithmetic mean of the numbers
\[
\Delta_{l+1},\ldots,\Delta_k.
\]

Similarly,
\[
\frac{\log E_n(A)-\log E_l(A)}{n-l}
=
\frac{\Delta_{l+1}+\cdots+\Delta_n}{n-l},
\]
which is the arithmetic mean of the larger collection
\[
\Delta_{l+1},\ldots,\Delta_n.
\]

Since
\[
\Delta_{l+1}\ge\Delta_{l+2}\ge\cdots\ge\Delta_n,
\]
the average of the first $k-l$ terms is greater than or equal to the average of all $n-l$ terms. Hence,
\[
\frac{\Delta_{l+1}+\cdots+\Delta_k}{k-l}
\ge
\frac{\Delta_{l+1}+\cdots+\Delta_n}{n-l},
\]
that is,
\[
\frac{\log E_k(A)-\log E_l(A)}{k-l}
\ge
\frac{\log E_n(A)-\log E_l(A)}{n-l}.
\]
This completes the proof.
\end{proof}
The following lemma gives a determinant bound for Hessian quotients, which is needed for the ABP estimate for Hessian quotient equations.

\begin{lemma}\label{quotient NM}
Let $A\geq0$, $0\leq l<k\leq n$, and let $C_{n,k,l}=\left(\binom nk/\binom nl\right)^{\frac{1}{k-l}}$. If $l=0$, then
\[
\sigma_k(A)^{\frac{1}{k}}\geq C_{n,k,0}\det(A)^\frac{1}{n}.
\]
If $l>0$ and $A>0$ satisfies $\lambda_{\max}(A)\leq\kappa\lambda_{\min}(A)$ for some $\kappa\geq1$, where $\lambda_{\max}(A),\lambda_{\min}(A)$ respectively are the maximum and minimum eigenvalues of $A$, then
\[
\left(\frac{\sigma_k(A)}{\sigma_l(A)}\right)^{\frac{1}{k-l}}
\geq C_{n,k,l}\kappa^{-\frac{l}{k-l}-\frac{n-1}{n}}\det(A)^{\frac{1}{n}}.
\]
\end{lemma}
\begin{proof}
The case $l=0$ is Lemma \ref{sigma-det}. For $l>0$, put $a=\lambda_{\min}(A)>0$. The eigenvalue bounds give
\[
\sigma_k(A)\geq\binom nk a^k,\qquad
\sigma_l(A)\leq\binom nl(\kappa a)^l.
\]
Also $\det(A)\leq\kappa^{n-1}a^n$. Combining these inequalities proves the assertion.
\end{proof}
\begin{remark}
A determinant lower bound without any further restriction is false in general when $l>0$. For example, for $n=2$, $k=2$, $l=1$ and $A=\operatorname{diag}(t,1)$, the quotient $\det(A)/\operatorname{tr}(A)$ remains bounded as $t\to\infty$, whereas $\det(A)^{1/2}\to\infty$.
\end{remark}

We now give the viscosity formulation for the Hessian quotient equations. We consider the equation
\begin{equation}\label{Hessian quotient equation}
F_{k,l}(-D^2u)
+
G(x,u,Du)
=
0 \quad \text{in} \quad \Omega
\end{equation}
where $F_{k,l}$ is defined in \eqref{F-kl-def}.
\begin{remark}
In our main result for this section, we assume that $|G(x,z,\xi)|\leq b(x)$ for some $b\in L^n$, so for $f\in L^n$, the class of problems posed above is equivalent to 
\[
F_{k,l}(-D^2u)
+
\tilde G(x,u,Du)
=
f(x).
\]
by the substitution $\tilde G = G+f$.
\end{remark}
A test function $\phi\in C^2(\Omega)$ touching $u$ from above at
$x_0$ is called $k$-admissible if
\[
\lambda(-D^2\phi(x_0))\in\mathcal C_k .
\]

For this subsection, a function $u\in USC(\Omega)$ is called a viscosity sub-solution on $k$-admissible upper tests
of \eqref{Hessian quotient equation} if, whenever $\phi$ is a
$k$-admissible test function touching $u$ from above at $x_0$, one has
\[
\left(
\frac{\sigma_k(-D^2\phi(x_0))}
{\sigma_l(-D^2\phi(x_0))}
\right)^{\frac1{k-l}}
+
G(x_0,u(x_0),D\phi(x_0))
\leq 0.
\]

\smallskip
Now, we state and prove the ABP estimate for Hessian quotient equations.

\begin{theorem}\label{ABP Hessian quotient}
Let $\Omega\subset\mathbb R^n$ be bounded and let $0\leq l<k\leq n$. Assume that $G$ is continuous and $|G(x,z,\xi)|\leq b(x)$ for some nonnegative $b\in L^n(\Omega)$. Let $u\in C(\overline\Omega)\cap C^{1,1}_{\mathrm{loc}}(\Omega)$ be a viscosity sub-solution on $k$-admissible upper tests of \eqref{Hessian quotient equation}. If $l>0$, suppose in addition that, for almost every point $x$ of the upper contact sets of $u$ at which $-D^2u(x)>0$,
\[
\lambda_{\max}(-D^2u(x))\leq\kappa\lambda_{\min}(-D^2u(x))
\]
with a constant $\kappa\geq1$ independent of the contact set. No such condition is needed for $l=0$. Then
\[
\sup_\Omega u^+\leq\sup_{\partial\Omega}u^+
+\frac{2^{(n-1)/n}\operatorname{diam}(\Omega)}{c_{n,k,l}\omega_n^{\frac{1}{n}}}
\|b\|_{L^n(\Omega)},
\]
where $c_{n,k,0}=\binom nk^{\frac{1}{k}}$, while $c_{n,k,l}=C_{n,k,l}\kappa^{-\frac{l}{k-l}-\frac{n-1}{n}}$ when $l>0$, and $\omega_n = |B_1|$ is the Lebesgue measure of the unit ball in $\mathbb R^n$.
\end{theorem}
\begin{proof}
If $M(u)=0$, the conclusion is immediate, where $M(u)$ is defined in \eqref{def-M}. Otherwise, use the compact contact sets $C_r$, $0<r<M(u)$, in the proof of Theorem \ref{G-ABP-thm}. At almost every point of these sets, $-D^2u\geq0$. Points at which $\det(-D^2u)=0$ contribute nothing to the area integral. At each remaining point $-D^2u>0$, and the Taylor polynomial of $u$ with $\varepsilon|y-x|^2$ added is an admissible test function for sufficiently small $\varepsilon>0$. Letting $\varepsilon\downarrow0$ gives
\[
\left(\frac{\sigma_k(-D^2u)}{\sigma_l(-D^2u)}\right)^{\frac{1}{k-l}}
\leq b(x).
\]
By Lemma \ref{quotient NM}, $c_{n,k,l}\det(-D^2u)^{\frac{1}{n}}\leq b$ on the contact set at almost every point with positive determinant. Applying the area formula to the locally Lipschitz map $Du$ gives
\[
|B_r|\leq c_{n,k,l}^{-n}\int_{C_r}b^n\,dx
\leq\frac{2^{n-1}}{c_{n,k,l}^n}
\|b\|_{L^n(\Omega)}^n.
\]
Letting $r$ increase to the normalized height $M(u)$ proves the estimate.
\end{proof}

\noindent
For $l>0$, the estimate above uses information on the contact Hessian in addition to admissibility. The pinching condition supplies the determinant lower bound in Lemma \ref{quotient NM}; it is not a consequence of admissibility alone.

\subsection{ABP Estimates for Generalized Mean Curvature Equations}
Let $A:\Omega\to\mathbb R^{n\times n}$ be a symmetric positive definite matrix function and let $B:\Omega\to\mathbb R^n$. To preserve ellipticity for anisotropic $A$, we use the symmetric coefficient matrix
\begin{equation}\label{curvature-matrix}
\mathcal A(x,\xi)=A(x)^{1/2}
\left[\frac{I_n}{\sqrt{1+|\xi|^2}}-
\frac{\xi\otimes\xi}{(1+|\xi|^2)^{3/2}}\right]A(x)^{1/2}, \quad x \in \Omega, \ \xi \in \R^n.
\end{equation}
Here $A^{1/2}$ is the symmetric positive definite square root of $A$. We consider
\begin{equation}\label{pmc eqn_}
-\operatorname{Tr}\bigl(\mathcal A(x,Du)D^2u\bigr)
+\frac{\langle B(x),Du\rangle}{\sqrt{1+|Du|^2}}
=f(x)\quad\text{in }\Omega.
\end{equation}
The normalized drift is bounded by $|B(x)|$. The matrix in square brackets is the Hessian of $\xi\mapsto\sqrt{1+|\xi|^2}$; below it is denoted by $J(\xi)$. When $A=I_n$ and $B=0$, \eqref{pmc eqn_} is the non-divergence form of the prescribed mean curvature equation
\begin{equation}\label{pmc eqn div}
-\operatorname{div}\left(\frac{Du}{\sqrt{1+|Du|^2}}\right)=f(x) \quad \text{for} \quad x \ \in \ \Omega.
\end{equation}
Thus \eqref{pmc eqn_} provides an elliptic anisotropic extension of this model. Related curvature equations are discussed in \cite{CC,BI,Dirr,Giga}.

\begin{remark}\label{matrix-product-obstruction}
For general positive definite $A$, the unsymmetrized product $AJ(\xi)$ in the original formulation need not define an elliptic operator, where
\[
J(\xi) = \frac{I_n}{\sqrt{1+|\xi|^2}}-
\frac{\xi\otimes\xi}{(1+|\xi|^2)^{3/2}}, \quad \xi \ \in \ \R^n.
\]
 Indeed, in dimension two, take
\[
A=\begin{pmatrix}1&0\\0&4\end{pmatrix},\qquad \xi=(3,3).
\]
Then
\[
\frac{AJ(\xi)+J(\xi)A}{2}
=\frac1{19^{3/2}}\begin{pmatrix}10&-45/2\\-45/2&40\end{pmatrix},
\]
whose determinant equals $-425/(4\cdot19^3)<0$.
Consequently, positivity of $A$ and $J$ alone does not justify a trace--determinant inequality for $AJ(-X)$ with $X\leq0$. The coefficient in \eqref{curvature-matrix} is symmetric positive definite and satisfies $\det\mathcal A=\det A\det J$. If $A$ and $J(\xi)$ commute, then $A^{1/2}J(\xi)A^{1/2}=AJ(\xi)$; in particular, the original scalar case $A=a(x)I_n$ is unchanged.
\end{remark}
We place the following assumptions on $A,B$ in \eqref{pmc eqn_}.
\begin{assumption}\label{A-assumption_}
Assume that the following conditions hold. 
\begin{enumerate}[label=\normalfont(\arabic*)]
    \item\label{A-assumption_1}
    \(A(x)\) is symmetric and non-negative definite for all \(x\in\Omega\), and $A, B \in C(\overline{\Omega})$.

    \item\label{A-assumption_2} There exists a locally integrable function
\(\mu:\mathbb{R}^n\to(0,\infty)\) such that

    \(1\le \mu(x)\det(A(x))\) for all \(x\in\Omega\).

\end{enumerate}
\end{assumption}
\noindent
We now prove a sup-norm estimate for viscosity sub-solutions of \eqref{pmc eqn_}. The smallness condition below reflects the finite mass of the gradient weight associated with the mean curvature operator; related estimates for divergence form equations are proved in \cite{ABP_divform}.
\begin{theorem}\label{PMC_main_result_}
Let $\Omega\subseteq \mathbb R^n$ be an open and bounded domain with boundary $\partial \Omega$, and let $A:\Omega\to \mathbb R^{n\times n}$ and $B:\Omega\to \mathbb R^{n}$ in \eqref{pmc eqn_} satisfy Assumption \ref{A-assumption_}. Suppose that $f\in C(\Omega)$ with $f^+\in L^n(\Omega,\mu)$ and $B\in L^n(\Omega,\mu)$, where $\mu$ is given in Assumption \ref{A-assumption_}, and assume the smallness condition:
\begin{equation}\label{small f}
\|f^+\|_{L^n(\Omega,\mu)} < \left(e^{K/\omega_n}-1\right)^{-1/n} \quad \text{with}\quad K=2^{n-1}\left(1+\|B\|_{L^n(\Omega,\mu)}^n\right),
\end{equation}
where $\omega_n$ denotes the Lebesgue measure of $B_1(0)\subset \R^n$.
Then, for any viscosity sub-solution $u\in C(\overline\Omega)\cap C^{1,1}_{\mathrm{loc}}(\Omega)$ to the equation \eqref{pmc eqn_}, we have
\begin{equation*}
\sup_\Omega u^+\le \sup_{\partial\Omega}u^+ +\textup{diam}(\Omega)\frac{\|f^+\|_{L^n(\Omega,\mu)}\left(e^{K/\omega_n}-1\right)^{\frac{1}{n}}}{\sqrt{1-\|f^+\|_{L^n(\Omega,\mu)}^2\left(e^{K/\omega_n}-1\right)^{2/n}}}.
\end{equation*}
\end{theorem}

\begin{proof}
The main idea of the proof is to apply Theorem \ref{G-ABP-thm} to the equation \eqref{pmc eqn_}. 
First assume $a=\|f^+\|_{L^n(\Omega,\mu)}>0$. The case $a=0$ is treated at the end of the proof. To begin with, we clarify the choices of functions and parameters in Assumption \ref{F-assumption} for this proof. In particular, we take $\mu$ being the one given in Assumption \ref{A-assumption_}, $\gamma_0 = 0,\  p_0 = n$, and let $g_1,g_2:\mathbb R^n\to\mathbb R$ be defined as
\[
g_1(\xi) = [\textsf{det}(J(\xi))]^\frac{1}{n},
\]
with $J:\mathbb R^n\to\mathbb R^{n\times n}$ being defined as
\begin{equation}\label{def J}
J( \xi) =\frac{I_n}{(1+|\xi|^2)^{1/2}}-\frac{\xi\otimes \xi}{(1+|\xi|^2)^{3/2}}, \quad \xi \ \in \ \R^n
\end{equation}
and
\begin{equation}\label{g2}
g_2 (\xi)= \left(\frac{|\xi|^n}{(|\xi|^2+1)^\frac{n}{2}}+\|f^+\|_{L^n(\Omega,\mu)}^n\right)^{-1/n}, \quad \xi \in \mathbb{R}^n.
\end{equation}
Here, the boundedness of $g_1$ and $g_2$ follows from equation \eqref{def J} and the assumption $0<\|f^+\|_{L^n(\Omega,\mu)}$, respectively. We take the following $F$ and $G$, which are independent of the $z$-variable, for equation \eqref{pmc eqn_}:
\begin{align*}
&\begin{array}{ccl}
F: & \Omega\times\mathbb{R}^n\times S(n)&\longrightarrow\mathbb R\\
&(x,\xi,X) &\longmapsto -\textup{Tr}\big(\mathcal A(x,\xi)X\big),
\end{array}\\
&\begin{array}{ccl}
G: & \Omega\times\mathbb{R}^n&\longrightarrow\mathbb R\\
&(x,\xi) &\longmapsto \frac{\langle B(x),\xi\rangle}{(1+|\xi|^2)^{1/2}}.
\end{array}
\end{align*}
Based on these choices and the Assumption~\ref{A-assumption_}, we start to show equation \eqref{pmc eqn_} (i.e., the functions $F$ and $G$ defined above) satisfies \textup{Assumption \ref{F-assumption}}.

\smallskip
The operators $F$ and $G$ are continuous. Let $\phi\in C^2(\Omega)$ and $x\in\Omega$ with $Y=-D^2\phi(x)\geq0$. By Lemma \ref{J-eigenvalues},
\begin{equation}\label{det J}
\det J(\xi)=(1+|\xi|^2)^{-(n+2)/2}.
\end{equation}
Put $P=\mathcal A(x,D\phi(x))$. The matrix $P^{1/2}YP^{1/2}$ is symmetric and non-negative definite. The arithmetic--geometric mean inequality for its eigenvalues gives
\[
\operatorname{Tr}(PY)=\operatorname{Tr}(P^{1/2}YP^{1/2})
\geq n\det(P)^{\frac{1}{n}}\det(Y)^{\frac{1}{n}}.
\]
Since $\det P=\det A(x)\det J(D\phi(x))$ and $\mu(x)\det A(x)\geq1$, we obtain
\[
\mu(x)^{\frac{1}{n}}F(x,D\phi(x),D^2\phi(x))
\geq n\det J(D\phi(x))^{\frac{1}{n}}\det(-D^2\phi(x))^{\frac{1}{n}}.
\]
This proves Assumption \ref{F-assumption}-\ref{F-assumption-2} for the chosen $g_1$ (and in fact for $ng_1$).
Additionally, we observe for $(x,\xi)\in\Omega\times (\mathbb R^n\setminus\{0\})$,
\begin{align*}
g_2(\xi)\left|G(x,\xi)\right|
&=\left(\frac{|\xi|^n}{(|\xi|^2+1)^\frac{n}{2}}+\|f^+\|_{L^n(\Omega,\mu)}^n\right)^{-\frac{1}{n}}\left|\frac{\langle B(x),\xi\rangle}{\sqrt{1+|\xi|^2}}\right|\\
&\le \frac{\sqrt{|\xi|^2+1}}{|\xi|}\left(\frac{|\xi||B(x)|}{\sqrt{1+|\xi|^2}}\right) = |B(x)|.
\end{align*}
At $\xi=0$, the same inequality holds since $G(x,0)=0$. Thus, \textup{Assumption \ref{F-assumption}}-\ref{F-assumption-3} is also satisfied by choosing $b(x)=|B(x)|$ and from \textup{Assumption \ref{A-assumption_}}-\ref{A-assumption_1}, and we have verified that \eqref{pmc eqn_} satisfies \textup{Assumption \ref{F-assumption}}.

\smallskip
We next evaluate the value of $\gamma$ defined in \eqref{gamma-def}. Recall that we choose $p_0=n$, $\gamma_0=0$, and $b(x)=|B(x)|$ in this proof. Using the choice of $g_2$, we see that
\[
\|g_2\|_{L^\infty(\mathbb{R}^n)} \|f^+\|_{L^{n}(\Omega, \mu)}=1.
\]
Thus, we have
\begin{equation}\label{gamma K}
\gamma=2^{\frac{n-1}{n}}\left(1+\|B\|^n_{L^n(\Omega, \mu)}\right)^{\frac{1}{n}}=K^{\frac{1}{n}}.
\end{equation}
On the other hand, using \eqref{det J} and by a direct computation, for any $r>0$, we have
\begin{equation}\label{g1g2 norm}
\begin{aligned}
&\|g_1g_2\|_{L^n(B_r(0))}^n = \int_{B_r(0)}\frac{\textsf{det}(J(\xi))}{\frac{|\xi|^n}{(1+|\xi|^2)^\frac{n}{2}}+\|f^+\|_{L^n(\Omega,\mu)}^n}\text{d}\xi \\
&= \int_{B_r(0)}\frac{1}{\left((1+|\xi|^2)|\xi|^n+(1+|\xi|^2)^{1+\frac{n}{2}}\|f^+\|_{L^n(\Omega,\mu)}^n\right)}\text{d}\xi \\
&= \omega_n\log\left(\left(\frac{r}{\sqrt{r^2+1}}\right)^n+\|f^+\|_{L^n(\Omega,\mu)}^n\right)-\omega_n\log(\|f^+\|_{L^n(\Omega,\mu)}^n)\\
&= \omega_n\log\left(1+\left(\tfrac{r}{\sqrt{r^2+1}}\right)^n\|f^+\|_{L^n(\Omega,\mu)}^{-n}\right),
\end{aligned}
\end{equation}
For the radial integration, use $t=\rho^n(1+\rho^2)^{-n/2}$, so $dt=n\rho^{n-1}(1+\rho^2)^{-(n+2)/2}d\rho$. Here $\omega_n=|B_1(0)|$. Letting $r\rightarrow \infty$, it follows that
\[
\|g_1g_2\|_{L^n(\R^n)}=\left[\omega_n\log \left(1+\|f^+\|_{L^n(\Omega, \mu)}^{-n}\right)\right]^{\frac{1}{n}}.
\]
Using the smallness condition \eqref{small f} on $\|f^+\|_{L^n(\Omega, \mu)}$, it is not hard to check that 
\[
\gamma< \|g_1g_2\|_{L^n(\R^n)}.
\]
Thus, we can apply Theorem \ref{G-ABP-thm}.

\smallskip
We define the function $H:\left(0, \|g_1g_2\|_{L^n(\mathbb R^n)}\right)\to(0,\infty)$ following the way in \eqref{H-def}. Explicitly, from \eqref{g1g2 norm}, we have
\[
H(s)=\frac{\|f^+\|_{L^n(\Omega,\mu)}\left(e^{s^n/\omega_n}-1\right)^{\frac{1}{n}}}{\sqrt{1-\|f^+\|_{L^n(\Omega,\mu)}^2\left(e^{s^n/\omega_n}-1\right)^{2/n}}}.
\]
Then, recall that $\gamma_0=0$ in this proof, it follows from Theorem \ref{G-ABP-thm} and \eqref{gamma K} that 
\begin{equation*}
\sup_\Omega u^+\le \sup_{\partial\Omega}u^+ +\textup{diam}(\Omega)\frac{\|f^+\|_{L^n(\Omega,\mu)}\left(e^{K/\omega_n}-1\right)^{\frac{1}{n}}}{\sqrt{1-\|f^+\|_{L^n(\Omega,\mu)}^2\left(e^{K/\omega_n}-1\right)^{2/n}}},
\end{equation*}
where $K$ is defined in \eqref{small f}. If $a=0$, apply the proved estimate to the continuous source $f_\varepsilon=f^++\varepsilon$. The function $u$ is still a sub-solution since $f\leq f_\varepsilon$. The measure $\mu(\Omega)$ is finite because $\Omega$ is bounded and $\mu\in L^1_{\mathrm{loc}}(\mathbb R^n)$. Hence $\|f_\varepsilon\|_{L^n(\Omega,\mu)}\to0$. For sufficiently small $\varepsilon$, the strict smallness condition holds, and the right-hand side of the height estimate tends to zero. This proves the stated conclusion also when $a=0$.
\end{proof}

The following remark shows that the estimate in Theorem \ref{PMC_main_result_} can be improved when $B=0$.

\begin{remark}
When $B=0$, the same contact-set argument gives a sharper estimate without the factor $2^{n-1}$. Put
\[
\theta=\frac{\|f^+\|_{L^n(\Omega,\mu)}}{n\omega_n^{\frac{1}{n}}}.
\]
If $\theta<1$, then
\[
\sup_\Omega u^+\leq\sup_{\partial\Omega}u^+
+\operatorname{diam}(\Omega)\frac{\theta}{\sqrt{1-\theta^2}}.
\]
Indeed, the matrix inequality in the proof gives $n^n\det J(Du)\det(-D^2u)\leq\mu(f^+)^n$ on the contact set. Since
\[
\int_{B_r}\det J(\xi)\,d\xi
=\omega_n\frac{r^n}{(1+r^2)^{n/2}},
\]
the area formula yields $r/\sqrt{1+r^2}\leq\theta$ for every $r<M(u)$. Inverting this inequality proves the claim, including the zero-source case. For $A=I$ and $\mu=1$, this recovers the mean-curvature estimate in \cite{ABP_divform}, with the norm over $\Omega$ in place of the contact-set norm used there; see also \cite[Corollary 10.6]{GT}. The formulation here includes a variable anisotropic coefficient, while the viscosity result in \cite{ABP_divform} does not require locally Lipschitz gradients.
\end{remark}

The following lemma we used in the proof of Theorem \ref{PMC_main_result_} is from a direct matrix computation, but we include it here for completeness.

\begin{lemma}\label{J-eigenvalues}
Let $J:\mathbb R^n\to\mathbb R^{n\times n}$ be defined as in \eqref{def J}. Then $J(\xi)$ has eigenvalues $(1+|\xi|^2)^{-1/2}$ with geometric multiplicity $n-1$ and $(1+|\xi|^2)^{-3/2}$ with multiplicity $1$ when $\xi\ne0$. At $\xi=0$, $J(0)=I_n$, so the eigenvalue $1$ has multiplicity $n$. 
\end{lemma}

\begin{proof} At $\xi=0$ the assertion follows from $J(0)=I_n$. Suppose $\xi\ne0$. The proof follows from direct calculation, and we provide details here for completeness. Rearranging \eqref{def J}, we have
\[
J(\xi) = \frac{1}{(1+|\xi|^2)^{1/2}}\left(I-\frac{\xi\otimes \xi}{1+|\xi|^2}\right).
\]
Note that for any $\alpha\in\mathbb R^n$ in the $n-1$ dimensional subspace orthogonal to $\xi$, 
\[
\frac{\xi\otimes\xi}{1+|\xi|^2}\alpha = \frac{\xi\langle\xi,\alpha\rangle}{1+|\xi|^2} = 0.
\]
Then, let $\alpha_1, \alpha_2, ...,\alpha_{n-1}$ be an orthonormal basis for this subspace. We can extend this to an orthonormal basis of $\mathbb{R}^n$ by adding $\hat{\xi}$, where $\hat{\xi}=\xi/|\xi|$ is a unit vector in the direction $\xi$. Thus,
\[
\frac{\xi\otimes\xi}{1+|\xi|^2}\hat\xi = \frac{\xi\langle\xi,\hat\xi\rangle}{1+|\xi|^2} = \frac{|\xi|^2}{1+|\xi|^2}\hat\xi.
\]
From the last two formulas, we see that there exists an orthonormal matrix $H$ such that
\[
H\left(\frac{\xi\otimes\xi}{1+|\xi|^2}\right)H^{-1} = \begin{bmatrix}
0 &\cdots&0  & 0\\
 \vdots& \ddots &\vdots &\vdots\\
 0&\cdots & 0 &0\\
 0& \cdots&0 & \frac{|\xi|^2}{1+|\xi|^2}
 \end{bmatrix}.
 \]
 Thus,
 \begin{align*}
 HJ( \xi)H^{-1} &=  \frac{1}{(1+|\xi|^2)^{1/2}}H\left(I-\frac{\xi\otimes \xi}{1+|\xi|^2}\right)H^{-1}\\
 &= \frac{1}{(1+|\xi|^2)^{1/2}}\left(I-H\left(\frac{\xi\otimes \xi}{1+|\xi|^2}\right)H^{-1}\right)\\
 &= \frac{1}{(1+|\xi|^2)^{1/2}}\begin{bmatrix}1 &\cdots&0  & 0\\
 \vdots& \ddots &\vdots &\vdots\\
 0&\cdots & 1 &0\\
 0& \cdots&0 & \frac{1}{1+|\xi|^2}\end{bmatrix}.
 \end{align*}
 The diagonal elements of $HJ( \xi)H^{-1}$ are exactly the eigenvalues of $J( \xi)$ as $H$ is a similarity transformation. Thus, the lemma is proved.
\end{proof}

\subsection{ABP Estimates for Mean Curvature Equations with Quadratic Gradient Growth}
We next include a normalized quadratic term. Its growth is at most linear as $|Du|$ tends to infinity; the numerator is quadratic in $Du$.
\begin{theorem}\label{quadratic-mean-curvature}
Let $A$ and $\mu$ satisfy Assumption \ref{A-assumption_}, let $B\in C(\overline\Omega)\cap L^n(\Omega,\mu)$, and let $C:\Omega\times\mathbb R\to\mathbb R^{n\times n}$ be continuous. Assume that $c\geq0$ belongs to $L^n(\Omega,\mu)$ and
\[
|\langle\xi,C(x,z)\xi\rangle|\leq c(x)|\xi|^2
\quad\text{for all }(x,z,\xi)\in\Omega\times\mathbb R\times\mathbb R^n.
\]
Suppose $f\in C(\Omega)$, $f^+\in L^n(\Omega,\mu)$, and $u\in C(\overline\Omega)\cap C^{1,1}_{\mathrm{loc}}(\Omega)$ is a viscosity sub-solution of
\[
-\operatorname{Tr}(\mathcal A(x,Du)D^2u)
+\frac{\langle B(x),Du\rangle}{\sqrt{1+|Du|^2}}
+\frac{\langle Du,C(x,u)Du\rangle}{\sqrt{1+|Du|^2}}=f(x).
\]
Set
\[
S=2^{n-1}\left(\|f^+\|_{L^n(\Omega,\mu)}^n
+\||B|+c\|_{L^n(\Omega,\mu)}^n\right)
\]
and define
\[
Q(r)=n^{n+1}\omega_n\int_0^r
\frac{\rho^{n-1}}{(1+\rho^2)^{(n+2)/2}(1+\rho)^n}\,d\rho.
\]
The function $Q$ is strictly increasing on $[0,\infty)$. We use the inverse
\[
Q^{-1}:[0,Q(\infty))\longrightarrow[0,\infty),\qquad Q^{-1}(0)=0.
\]
If $S<Q(\infty)$, then
\[
\sup_\Omega u^+\leq\sup_{\partial\Omega}u^+
+\operatorname{diam}(\Omega)Q^{-1}(S).
\]
\end{theorem}
\begin{proof}
Take $p_0=n$, $\gamma_0=0$, $g_1(\xi)=n\det(J(\xi))^{\frac{1}{n}}$, and $g_2(\xi)=(1+|\xi|)^{-1}$. For $X\leq0$, put $P=\mathcal A(x,\xi)$. Apply the trace--determinant inequality to $P^{1/2}(-X)P^{1/2}$. Together with $\mu\det A\geq1$, this gives Assumption \ref{F-assumption}-\ref{F-assumption-2}, as in the proof of Theorem \ref{PMC_main_result_}. For the lower-order terms, we have
\[
\begin{split}
g_2(\xi)\left|\frac{\langle B(x),\xi\rangle}{\sqrt{1+|\xi|^2}}
+\frac{\langle\xi,C(x,z)\xi\rangle}{\sqrt{1+|\xi|^2}}\right|
&\leq |B(x)|+c(x)\frac{|\xi|^2}{(1+|\xi|)\sqrt{1+|\xi|^2}}\\
&\leq |B(x)|+c(x).
\end{split}
\]
Polar integration using \eqref{det J} gives
\[
Q(r)=\|g_1g_2\|_{L^n(B_r)}^n.
\] It is strictly increasing; its integrand is comparable to $\rho^{n-1}$ near zero and to $\rho^{-n-3}$ at infinity. Thus $0<Q(\infty)<\infty$, and we set $Q^{-1}(0)=0$. Since $\|g_2\|_\infty=1$, the $n$th power of $\gamma$ in \eqref{gamma-def} equals $S$. Theorem \ref{G-ABP-thm} gives the result.
\end{proof}

\section*{Acknowledgement}
J. Fang was partially supported by Spike Tickle STEM Endowed Fellowship (2025). J. Peltier was partially supported by National Merit Scholarship (2022) and Barry Goldwater Scholarship (2025). The preliminary results of this work are presented in J. Peltier's honors thesis, submitted to the University of Tennessee.  T. Phan was partially supported by Simons Foundation, grant \# 769369.


\end{document}